\documentclass[12pt,reqno]{amsart} 
\usepackage{amssymb}
\usepackage{mathrsfs}
\usepackage{enumerate}
\usepackage[all]{xy}
\usepackage[usenames]{color}
\usepackage[colorlinks=true, citecolor=blue, linkcolor=blue]{hyperref}

\newcommand{\xxx}[2][k]{#2_1\times\cdots\times #2_{#1}}

\newcommand{\boldd}[1]{\emph{#1}\medskip\\} 

\let\Gamma=\varGamma
\let\Delta=\varDelta

\newcommand{\6}[1]{_{\le #1}}
\newcommand{\7}[1]{_{\le S_{#1}}}   
\newcommand{\8}[1]{^{(#1)}}     

\renewenvironment{enumerate}[1][]
{\begin{enumerat}[#1]\setlength{\itemsep}{6pt}}{\end{enumerat}}

\newenvironment{enuma}{\begin{enumerate}[{\rm(a) }]}{\end{enumerate}}
\newenvironment{enumi}{\begin{enumerate}[{\rm(i) }]}{\end{enumerate}}

\renewenvironment{itemize}
{\begin{itemiz}\setlength{\itemsep}{6pt} \setlength{\itemindent}{-20pt} }
{\end{itemiz}}

\definecolor{darkgreen}{rgb}{0,0.5,0}
\definecolor{bluegreen}{rgb}{0,0.2,0.8}
\definecolor{darkred}{rgb}{0.8,0,0}
\definecolor{newercolor}{rgb}{0.2,0,1}
\definecolor{darkyellow}{rgb}{0.7,0.7,0}
\definecolor{orange}{rgb}{1.0,0.5,0}
\definecolor{darkorange}{rgb}{0.8,0.4,0}

\numberwithin{table}{section}

\newlength{\short}
\newcommand{\4}[1]{\widebar{#1}}
\newcommand{\5}[1]{\widehat{#1}}
\newcommand{\9}[1]{{}^{#1}\!}

\def\pair[#1,#2]{[\hskip-1.5pt[#1,#2]\hskip-1.5pt]}

\SelectTips{cm}{10} \UseTips   

\let\oldcirc=\circ
\renewcommand{\circ}{\mathchoice
    {\mathbin{\scriptstyle\oldcirc}}{\mathbin{\scriptstyle\oldcirc}}
    {\mathbin{\scriptscriptstyle\oldcirc}}
    {\mathbin{\scriptscriptstyle\oldcirc}}}

\def\beq#1\eeq{\begin{equation*}#1\end{equation*}}
\def\beqq#1\eeqq{\begin{equation}#1\end{equation}}

\numberwithin{equation}{section}

\newtheorem{Thm}{Theorem}[section]
\newtheorem{Prop}[Thm]{Proposition}

\newtheorem{Lem}[Thm]{Lemma}
\newtheorem{Defi}[Thm]{Definition}

\newtheorem{Hyp}[Thm]{Hypotheses}
\newtheorem{Not}[Thm]{Notation}
\newtheorem{Ex}[Thm]{Example}

\theoremstyle{definition}
\newtheorem{Rmk}[Thm]{Remark}

\newtheorem{Th}{Theorem}

\newcommand{\widebar}[1]
      {\overset{{\mskip3mu\leaders\hrule height0.4pt\hfill\mskip3mu}}{#1}
      \vphantom{#1}}

\newcounter{let} 
\loop\stepcounter{let}
\expandafter\edef\csname cal\alph{let}\endcsname%
{\noexpand\mathcal{\Alph{let}}}
\ifnum\thelet<26\repeat

\loop\stepcounter{let}
\expandafter\edef\csname scr\alph{let}\endcsname%
{\noexpand\mathscr{\Alph{let}}}
\ifnum\thelet<26\repeat

\newcommand{\tdef}[2][]{\expandafter\newcommand\csname#2\endcsname%
{#1\textup{#2}}}
\tdef{Iso}   \tdef{Aut}    \tdef{Out}    \tdef{Inn}    \tdef{Hom}
\tdef{End}   \tdef{Inj}    \tdef{map}    \tdef{Ker}    \tdef{Ob}
\tdef{Mor}   \tdef{Res}    \tdef{Id}     \tdef{Fr}     \tdef{Spin} 
\tdef{rk}    \tdef{conj}   \tdef{incl}   \tdef{proj}   \tdef{diag} 
\tdef{trf}   \tdef{Sol}    \tdef{He}     \tdef{Sz}     \tdef{cj}
\tdef{Rep}   \tdef{pr}    \tdef{Inndiag} \tdef{Outdiag}  \tdef{expt}
\tdef{supp}  \tdef{Isom}   \tdef{ord}    \tdef{Coker}   \tdef{Tr}
\tdef[_]{typ} \tdef[^]{op} \tdef[^]{ab}   \tdef{lcm}  \tdef{McL}
\tdef{restr}  \tdef{Comp}  \tdef[_]{fn}

\newcommand{\fdef}[1]{\expandafter\newcommand\csname#1\endcsname%
{\mathfrak{#1}}}
\fdef{X}  \fdef{red}  \fdef{foc}  \fdef{hyp}  \fdef{Lie} \fdef{Y}

\newcommand{\bbdef}[1]{\expandafter\newcommand%
\csname#1\endcsname{\mathbb{#1}}}
\bbdef{C} \bbdef{F} \bbdef{R} \bbdef{Z} \bbdef{N} \bbdef{Q} \bbdef{K}

\newcommand{\itdef}[1]{\expandafter\newcommand\csname#1\endcsname%
{\textit{#1}}}
\itdef{PSL}  \itdef{PSU}  \itdef{SL}  \itdef{SU}  \itdef{GL} \itdef{GU}
\itdef{Sp}   \itdef{PSp} \itdef{PSO} \itdef{SO}   \itdef{SD} \itdef{PGU} 
\itdef{PGL}  \itdef{Co}  \itdef{Fi}  \itdef{GO}

\newcommand{\sminus}{\smallsetminus}
\newcommand{\lie}[3]{\def\test{#2}\def\tst{G}\ifx\test\tst{{}^{#1}#2_{#3}}
\else{{}^{#1}\!#2_{#3}}\fi}
\renewcommand{\*}{\,\lower6pt\hbox{\Large{\textup{*}}}\,}
\newcommand{\syl}[2]{\textup{Syl}_{#1}(#2)}
\newcommand{\sylp}[1]{\syl{p}{#1}}

\newcommand{\autf}{\Aut_{\calf}}

\newcommand{\outf}{\Out_{\calf}}

\newcommand{\homf}{\Hom_{\calf}}
\newcommand{\isof}{\Iso_{\calf}}
\newcommand{\defeq}{\overset{\textup{def}}{=}}

\newcommand{\mxfoura}[8]{\left(\begin{smallmatrix}#1&#2&#3&#4\\#5&#6&#7&#8}
\newcommand{\mxfourb}[8]{\\#1&#2&#3&#4\\#5&#6&#7&#8\end{smallmatrix}\right)}

\let\emptyset=\varnothing
\renewcommand{\:}{\colon}

\newcommand{\nsg}{\trianglelefteq}

\newcommand{\snsg}{\nsg\,\nsg}

\newcommand{\til}[1]{\widetilde{#1}}
\let\too=\longrightarrow
\let\xto=\xrightarrow

\newcommand{\gen}[1]{{\langle}#1{\rangle}}
\newcommand{\Gen}[1]{{\bigl\langle}#1{\bigr\rangle}}

\newcommand{\longleft}[1]{\;{\leftarrow%
\count255=0 \loop \mathrel{\mkern-6mu}%
    \relbar\advance\count255 by1\ifnum\count255<#1\repeat}\;}
\newcommand{\longright}[1]{\;{\count255=0 \loop \relbar\mathrel{\mkern-6mu}%
    \advance\count255 by1\ifnum\count255<#1\repeat\rightarrow}\;}
\newcommand{\Right}[2]{\overset{#2}{\longright#1}}
\newcommand{\RIGHT}[3]{\mathrel{\mathop{\kern0pt\longright#1}
        \limits^{#2}_{#3}}}

\newcommand{\LEFT}[3]{\mathrel{\mathop{\kern0pt\longleft#1}\limits^{#2}_{#3}}
}
\newcommand{\dRIGHT}[3]{\mathrel{%
   \mathop{\vcenter{\baselineskip=0pt\hbox{$\kern0pt\longright#1$}%
   \hbox{$\kern0pt\longright#1$}}}\limits^{#2}_{#3}}}
\newcommand{\LRIGHT}[3]{\mathrel{%
   \mathop{\vcenter{\baselineskip=0pt\hbox{$\kern0pt\longleft#1$}%
   \hbox{$\kern0pt\longright#1$}}}\limits^{#2}_{#3}}}
\newcommand{\RLEFT}[3]{\mathrel{%
   \mathop{\vcenter{\baselineskip=0pt\hbox{$\kern0pt\longright#1$}%
   \hbox{$\kern0pt\longleft#1$}}}\limits^{#2}_{#3}}}
\newcommand{\onto}[1]{\;{\count255=0 \loop \relbar\mathrel{\mkern-6mu}%
    \advance\count255 by1
    \ifnum\count255<#1 \repeat \twoheadrightarrow}\;}

\title{Normalizers of sets of components in fusion systems}

\author{Bob Oliver}
\address{Universit\'e Sorbonne Paris Nord, LAGA, UMR 7539 du CNRS, 
99, Av. J.-B. Cl\'ement, 93430 Villetaneuse, France.}
\email{bobol@math.univ-paris13.fr}
\thanks{B. Oliver is partially supported by UMR 7539 of the CNRS}


\begin{document}

\begin{abstract} 
We describe some new ways to construct saturated fusion subsystems, 
including, as a special case, the normalizer of a set of components of the 
ambient fusion system. This was motivated in part by Aschbacher's 
construction of the normalizer of one component, and in part by joint work 
with three other authors where we had to construct the normalizer of all of 
the components. 
\end{abstract}

\maketitle


A saturated fusion system over a finite $p$-group $S$ is a category whose 
objects are the subgroups of $S$, whose morphisms are injective 
homomorphisms between the subgroups, and which satisfy certain conditions 
first formulated by Puig (see \cite{Puig} and Definition \ref{d:fusion}). The 
motivating examples are the fusion systems $\calf_S(G)$ when $G$ is a 
finite group and $S\in\sylp{G}$: the objects of $\calf_S(G)$ are the 
subgroups of $S$, and the morphisms are those homomorphisms between 
subgroups induced by conjugation in $G$. 

By analogy with finite groups, Aschbacher, in \cite[Chapter 9]{A-gfit}, 
defined a component of a saturated fusion system $\calf$ over $S$ to be a 
subnormal fusion subsystem of $\calf$ that is quasisimple (see Definition 
\ref{d:component}). He then showed that the components of $\calf$ commute 
and satisfy other properties satisfied by the components of a finite group. 
In a later paper \cite[\S\,2.1]{A-comp.type}, he constructed the normalizer 
of a component; i.e., the unique largest fusion subsystem that contains the 
given component as a normal subsystem. 

In a recent paper with Carles Broto, Jesper M\o{}ller, and Albert Ruiz 
\cite{BMOR}, we needed to construct a normalizer for all of the components; 
i.e., a largest subsystem of $\calf$ that contains each of the components 
of $\calf$ as a normal subsystem. This turned out to be somewhat simpler 
than Aschbacher's construction of the normalizer of one component, but it 
also led this author to try to better understand Aschbacher's construction, 
and to look for ways in which it could be generalized. This has resulted in a 
slightly more explicit construction of these normalizers, and 
led to our two main Theorems \ref{t:H<G} and \ref{t:NJ}. In particular, 
Theorem \ref{t:H<G} provides a very general method for constructing 
saturated fusion subsystems (normal or not) of a given fusion system: one 
which we hope will have other applications in the future.

As one special case of Theorem \ref{t:NJ}, we get the following theorem 
about normalizers of components:

\begin{Th} \label{ThA}
Let $\calf$ be a saturated fusion system over a finite $p$-group $S$, and 
let $\calc_1,\dots,\calc_k\le\calf$ be the components of $\calf$ over 
$T_1,\dots,T_k\le S$. Let $J\subseteq\{1,\dots,k\}$ be a nonempty subset 
such that the subgroup $T_J=\gen{T_j\,|\,j\in J}$ is fully normalized in 
$\calf$ (Definition \ref{d:fusion}(c)), and let $\calc_J\le\calf$ be the 
central product of the components $\calc_j$ for $j\in J$. Set 
$N_J=N_S(T_J)$ and $W_J=\bigcap_{j\in J}N_S(T_j)$. Then there are saturated 
fusion subsystems $\calw_J\nsg\caln_J\le\calf$ over $W_J\nsg N_J\le S$ such 
that $\caln_J$ is the largest saturated fusion subsystem of $\calf$ 
containing $\calc_J$ as a normal subsystem, and $\calw_J$ is the largest 
saturated subsystem containing each $\calc_j$ (for $j\in J$) as a normal 
subsystem. 
\end{Th}

In other words, $\calc_J\nsg\caln_J$, and for each saturated fusion 
subsystem $\cald\le\calf$ such that $\calc_J\nsg\cald$ we have 
$\cald\le\caln_J$. Similarly, $\calc_j\nsg\calw_J$ for each $j\in J$, and 
for each $\cald\le\calf$ such that $\calc_j\nsg\cald$ for all $j\in J$, we 
have $\cald\le\calw_J$. 
More generally, Theorem \ref{t:NJ} says that a similar conclusion holds if 
the components are replaced by an arbitrary set of fusion subsystems 
satisfying certain conditions listed in Hypotheses \ref{h:Ei<F}. 

We mostly regard Theorems \ref{t:H<G} and \ref{t:NJ} as new tools for 
constructing saturated fusion subsystems of a given fusion system, with 
Theorem \ref{ThA} as one special case of particular interest. But they also 
have consequences for fundamental groups of linking systems associated to 
the fusion systems, and the subsystems constructed here can be associated 
to covering spaces of the geometric realizations of those linking systems. 
See Remark \ref{rmk:linking} at the end of Section 2 for a more detailed 
discussion.

In Section 1, we recall some basic definitions and properties of saturated 
fusion systems, all well known except (perhaps) for the technical Lemma 
\ref{l:f.n.=>recept}. In Section 2, we describe a very general way to 
construct saturated subsystems of a given saturated fusion system $\calf$ 
(see Hypotheses \ref{h:chi-setup} and Theorem \ref{t:H<G}). Then, in 
Section 3, we consider sets of commuting subsystems of a saturated fusion 
system $\calf$, under assumptions (Hypotheses \ref{h:Ei<F}) that include 
the case where the subsystems are the components of $\calf$ (Proposition 
\ref{p:Comp(F)}). This leads to Theorem \ref{t:NJ}, a special case of 
Theorem \ref{t:H<G}, where we construct normalizers of these subsystems. 
Theorem \ref{ThA} then follows as a special case of Theorem \ref{t:NJ}. 

We would like to thank both referees of this paper for their many 
suggestions, leading to a number of improvements in the presentation.

\bigskip

\newcommand{\kk}{\underline{\boldsymbol{k}}} 

\textbf{Notation: } We write $c_x$ for a homomorphism defined via 
conjugation by $x$ ($c_x(g)=xgx^{-1}$), and $c_x^P\in\Aut(P)$ for its 
restriction to a subgroup $P$ when $x$ normalizes $P$. As usual, when 
$H\le G$ is a pair of groups and $x\in G$, we write $\9xH=c_x(H)=xHx^{-1}$ 
and $H^x=c_x^{-1}(H)=x^{-1}Hx$. When $G$ is a group and $P,Q\le G$, we 
let $\Hom_G(P,Q)\subseteq\Hom(P,Q)$ be the set of all homomorphisms of the 
form $c_x$ for $x\in G$ such that $\9xP\le Q$. 

Functions and morphisms are always composed from right to left. Also, 
\begin{itemize} 
\item $\scrs(G)$ is the set of subgroups of a group $G$;
\item $\kk=\{1,\dots,k\}$ for $k\ge1$; and 
\item $\calh\6T=\calh\cap\scrs(T)$ when $\calh$ is a set of subgroups of 
a group $S$ and $T\le S$. 
\end{itemize}
When $\calf$ is a fusion system over $S$ and $\calh\subseteq\scrs(S)$, we 
let $\calf^\calh\subseteq\calf$ denote the full subcategory whose set of 
objects is $\calh$.


\section{Background on fusion and linking systems}

We summarize here our basic terminology  when working with fusion systems. 
For a prime $p$, a \emph{fusion system} over a finite $p$-group $S$ is a 
category whose objects are the subgroups of $S$, and whose morphisms are 
injective homomorphisms between subgroups such that for each $P,Q\le S$:
\begin{itemize} 
	\item $\homf(P,Q)\supseteq \Hom_S(P,Q)$; and 
	\item for each $\varphi\in\homf(P,Q)$, 
	$\varphi^{-1}\in\homf(\varphi(P),P)$.
\end{itemize}
Here, $\homf(P,Q)$ denotes the set of morphisms in $\calf$ from 
$P$ to $Q$. We also write $\isof(P,Q)$ for 
the set of isomorphisms, $\autf(P)=\isof(P,P)$, and 
$\outf(P)=\autf(P)/\Inn(P)$. For $P\le S$ and $g\in S$, we set 
\[ P^\calf = \{\varphi(P)\,|\,\varphi\in\homf(P,S)\} 
\qquad\textup{and}\qquad 
g^\calf = \{\varphi(g)\,|\,\varphi\in\homf(\gen{g},S)\}  \]
(the sets of subgroups and elements \emph{$\calf$-conjugate} to $P$ and 
to $g$).

\begin{Defi} \label{d:fusion}
Let $\calf$ be a fusion system over a finite $p$-group $S$.
\begin{enuma}

\item A subgroup $P\le S$ is \emph{fully automized} in $\calf$ if 
$\Aut_S(P)\in\sylp{\autf(P)}$.

\item A subgroup $P\le S$ is \emph{receptive} in $\calf$ if each 
isomorphism $\varphi\in\isof(Q,P)$ in $\calf$ extends to a morphism 
$\4\varphi\in\homf(N_\varphi^\calf,S)$, where 
	\[ N_\varphi^\calf = \{ x\in N_S(Q) \,|\, \varphi c_x^Q\varphi^{-1} \in 
	\Aut_S(P) \}. \]

\item A subgroup $P\le S$ is \emph{fully normalized} \emph{(fully 
centralized)} in $\calf$ if $|N_S(P)|\ge|N_S(Q)|$ ($|C_S(P)|\ge|C_S(Q)|$) 
for each $Q\in P^\calf$. 

\item The fusion system $\calf$ is \emph{saturated} if each 
$\calf$-conjugacy class of subgroups of $S$ contains a member that is fully 
automized and receptive.
\end{enuma}
\end{Defi}

We will sometimes need to refer to the following criteria for a fusion 
system to be saturated.

\begin{Prop}[{\cite[Theorem 5.2]{RS} or \cite[Proposition I.2.5]{AKO}}] 
\label{p:BLO<=>RS}
A fusion system is saturated if and only if it satisfies the following two 
conditions:
\begin{itemize}

\item \emph{(Sylow axiom)} each subgroup $P\le S$ that is fully normalized 
in $\calf$ is also fully automized and fully centralized; and 

\item \emph{(extension axiom)} each $P\le S$ that is fully 
centralized in $\calf$ is also receptive. 

\end{itemize}
\end{Prop}

We often need to refer to the following types or classes of subgroups 
in a fusion system. 

\begin{Defi} \label{d:subgroups}
Let $\calf$ be a fusion system over a finite $p$-group $S$. For a subgroup 
$P\le S$, 
\begin{enuma} 

\item $P$ is \emph{$\calf$-centric} if $C_S(Q)\le Q$ for each $Q\in 
P^\calf$;

\item $P$ is \emph{$\calf$-radical} if $O_p(\outf(P))=1$; 


\item $P$ is \emph{strongly closed} in $\calf$ if for each $x\in P$, 
$x^\calf\subseteq P$; and 

\item $P$ is \emph{central} in $\calf$ if each morphism 
$\varphi\in\homf(Q,R)$ in $\calf$ extends to some 
$\4\varphi\in\homf(PQ,PR)$ such that $\4\varphi|_P=\Id_P$.

\end{enuma}
Let $\calf^{cr}\subseteq\calf^c$ denote the sets of $\calf$-centric 
$\calf$-radical subgroups, and $\calf$-centric subgroups, respectively. 
Let 
\begin{enuma} \setcounter{enumi}{4}

\item $Z(\calf)$ (the \emph{center} of $\calf$) be the (unique) largest 
subgroup central in $\calf$; and set 

\item $\foc(\calf)=\gen{x^{-1}y\,|\,x,y\in S,~ y\in x^\calf}$ (the 
\emph{focal subgroup} of $\calf$). 

\end{enuma}
\end{Defi}

In many cases, to prove saturation, it is not 
necessary to prove the axioms for all conjugacy classes of subgroups. 
If $\calf$ is a fusion system over a finite $p$-group $S$ and $\calh$ is a 
set of subgroups of $S$ closed under $\calf$-conjugacy, then 
\begin{itemize} 
\item $\calf$ is \emph{$\calh$-saturated} if each member of $\calh$ is 
$\calf$-conjugate to a subgroup that is fully automized and receptive; and 

\item $\calf$ is \emph{$\calh$-generated} if each morphism in $\calf$ is a 
composite of restrictions of morphisms between members of $\calh$.
\end{itemize}
Under certain conditions on $\calh$, these two conditions suffice to show 
that $\calf$ is saturated.

\begin{Prop} \label{p:H-sat-gen}
Let $\calf$ be a fusion system over a finite $p$-group $S$, and let 
$\calh\subseteq\scrs(S)$ be a nonempty set of subgroups closed under 
$\calf$-conjugacy. 
\begin{enuma} 

\item Assume that $\calf$ is $\calh$-generated and $\calh$-saturated, 
and that for each $P\in\calf^c\sminus\calh$ there is $Q\in P^\calf$ such 
that $\Out_S(P)\cap O_p(\outf(P))\ne1$. Then $\calf$ is saturated.

\item \textup{(Stancu's criterion)} Assume $\calh$ is closed under 
overgroups in $\scrs(S)$. Assume also that $S$ is fully automized in 
$\calf$, and that each $P\in\calh$ that is fully normalized in $\calf$ is 
also receptive in $\calf$. Then $\calf$ is $\calh$-saturated. 

\end{enuma}
\end{Prop}

\begin{proof} Point (a) is shown in \cite[Theorem 2.2]{BCGLO1} (see also 
the discussion in \cite[Theorem I.3.10]{AKO}). 

By \cite[Proposition I.9.3(c$\Rightarrow$a)]{AKO}, Stancu's criterion for 
saturation \cite{Stancu} implies that in Definition \ref{d:fusion}(d). 
Point (b), the corresponding implication for $\calh$-saturation, holds by 
the same argument whenever $\calh$ is closed under overgroups. 
\end{proof}

We will need the following version of Alperin's fusion theorem for fusion 
systems. 

\begin{Thm} [{\cite[Theorem A.10]{BLO2}}] \label{t:AFT}
If $\calf$ be a saturated fusion system over a finite $p$-group $S$, then 
each morphism in $\calf$ is a composite of restrictions of automorphisms of 
subgroups that are $\calf$-centric and $\calf$-radical.
\end{Thm}

If $\calf$ is a fusion system over a finite $p$-group $S$, and 
$\beta\:S\too T$ is an isomorphism of groups, then $\9\beta\calf$ denotes 
the fusion system over $T$ defined by setting 
	\[ \Hom_{\9\beta\calf}(P,Q) = \bigl\{ \beta\varphi\beta^{-1} \,\big|\, 
	\varphi\in\homf(\beta^{-1}(P),\beta^{-1}(Q)) \bigr\} \]
for all $P,Q\le T$. (Recall that we compose from right to left.) In terms 
of this notation, two fusion systems $\calf$ over $S$ and $\cale$ 
over $T$ are \emph{isomorphic} if there is an isomorphism of groups 
$\beta\:S\xto{~\cong~}T$ such that $\cale=\9\beta\calf$.

\begin{Defi}[{\cite[Definition I.6.1]{AKO}}] \label{d:Fnormal}
Fix a saturated fusion system $\calf$ over a finite $p$-group $S$, and let 
$\cale\le\calf$ be a saturated fusion subsystem over $T\le S$. Then $\cale$ 
is \emph{normal} in $\calf$ (denoted $\cale\nsg\calf$) if the following 
four conditions are satisfied: 
\begin{itemize}

\item $T$ is strongly closed in $\calf$; 

\item \emph{(invariance condition)} $\9\alpha\cale=\cale$ for each 
$\alpha\in\Aut_{\calf}(T)$; 

\item \emph{(Frattini condition)} for each $P\leq T$ and each $\varphi \in 
\Hom_{\calf}(P,T)$, there are $\alpha\in\Aut_{\calf}(T)$ and 
$\varphi_0\in\Hom_{\cale}(P,T)$ such that $\varphi=\alpha\circ\varphi_0$; 
and 
		
\item \emph{(extension condition)} each $\alpha\in\Aut_{\cale}(T)$ extends 
to an automorphism $\4{\alpha}\in\Aut_{\calf}(TC_S(T))$ such that 
$[\4{\alpha},C_S(T)]\defeq\gen{x^{-1}\4\alpha(x)\,|\,x\in C_S(T)}\leq Z(T)$.

\end{itemize}
\end{Defi}

The following elementary property of normal subsystems will be needed.

\begin{Lem} \label{l:Fnormal}
Let $\cale\nsg\calf$ be saturated fusion systems over $T\nsg S$. Let 
$\calf_0\le\calf$ be a saturated fusion subsystem over $S_0\le S$ such that 
$\calf_0\ge\cale$, and assume the extension condition holds for 
$\cale\le\calf_0$. Then $\cale\nsg\calf_0$.
\end{Lem}

\begin{proof} Since $\cale\nsg\calf$, the subgroup $T$ is strongly closed 
in $\calf$, and hence is also strongly closed in $\calf_0$. By 
\cite[Proposition I.6.4]{AKO} and since $\cale\nsg\calf$, the \emph{strong 
invariance condition} holds: for each pair of subgroups $P\le Q\le T$, each 
$\varphi\in\Hom_\cale(P,Q)$, and each $\psi\in\homf(Q,T)$, we have 
$\psi\varphi(\psi|_P)^{-1}\in\Hom_\cale(\psi(P),\psi(Q))$. 

The strong invariance condition for $\cale\le\calf_0$ follows immediately 
from that for $\cale\le\calf$. Hence by 
\cite[Proposition I.6.4]{AKO} again, $\cale\le\calf_0$ also satisfies the 
invariance and Frattini conditions. So if it also satisfies the extension 
condition, then $\cale\nsg\calf_0$.
\end{proof}

We will need to work with quotient fusion systems, but only in the special 
(and very simple) case where we divide by a central subgroup. 

\begin{Defi} \label{d:F/Z}
Let $\calf$ be a saturated fusion system over a finite $p$-group $S$, and 
assume $Z\le Z(\calf)$ is a central subgroup. 
\begin{enuma} 

\item Let $\calf/Z$ be the fusion system over $S/Z$ where for all 
$P,Q\le S$ containing $Z$, 
	\[ \Hom_{\calf/Z}(P/Z,Q/Z) = \bigl\{ \varphi/Z \,\big|\, 
	\varphi\in\homf(P,Q), ~ (\varphi/Z)(gZ)=\varphi(g)Z \bigr\} . \]

\item If $\cale\le\calf$ is a fusion subsystem over $T\le S$, 
then $Z\cale\le\calf$ is the fusion subsystem over $ZT$ where for $P,Q\le 
ZT$, 
	\[ \Hom_{Z\cale}(P,Q) = \{ \varphi\in\homf(P,Q) \,|\, 
	\varphi|_{P\cap T} \in \Hom_\cale(P\cap T,Q\cap T) \} . \]

\end{enuma}
\end{Defi}

If $\cale\nsg\calf$, then $Z\cale\le\calf$ is a special (and much more 
elementary) case of a construction of Aschbacher \cite[Theorem 
8.20]{A-gfit}.


\begin{Lem} \label{l:F/Z}
Let $\calf$ be a saturated fusion system over a finite $p$-group $S$, 
let $Z\le Z(\calf)$ be a central subgroup, and let $\cale\le\calf$ be a 
saturated fusion subsystem over $T\le S$. Then 
\begin{enuma} 

\item $\calf/Z$ and $Z\cale$ are both saturated; and 

\item if $Z(\cale)\le Z$, then $Z\cale/Z\cong\cale/Z(\cale)$. 

\end{enuma}
\end{Lem}

\begin{proof} The statement that $\calf/Z$ is saturated is a special case 
of \cite[Proposition 5.11]{Craven}. 

For each $P\le ZT$, $\Aut_{Z\cale}(P)\cong\Aut_\cale(P\cap T)$ and 
$\Aut_{ZT}(P)\cong\Aut_T(P\cap T)$. So $P$ is fully automized in $Z\cale$ 
if and only if $P\cap T$ is fully automized in $\cale$. A similar 
argument shows that $P$ is receptive in $Z\cale$ if and only if 
$P\cap T$ is receptive in $\cale$, and thus $Z\cale$ is saturated since 
$\cale$ is saturated.

Assume that $Z(\cale)\le Z$, and hence that $Z(\cale)=Z\cap T$. Let 
$\psi\:T/Z(\cale)\xto{~\cong~}ZT/Z$ be the natural isomorphism. By the 
definitions, a morphism $\varphi/Z\in\Hom(P/Z,Q/Z)$ (for 
$\varphi\in\Hom(P,Q)$ such that $\varphi(Z)=Z$) lies in $Z\cale/Z$ if and 
only if $\psi^{-1}(\varphi/Z)\psi=(\varphi|_{P\cap T})/Z(\cale)$ lies in 
$\cale/Z(\cale)$. So $Z\cale/Z=\9\psi(\cale/Z(\cale))$, and hence 
$\cale/Z(\cale)\cong Z\cale/Z$. 
\end{proof}

Note that Lemma \ref{l:F/Z}(b) is a very elementary case of the second 
isomorphism theorem for fusion systems (see \cite[Proposition 
5.16]{Craven}). 

The following technical lemma will be useful later when proving that 
certain fusion systems are $\calh$-saturated.

\begin{Lem} \label{l:f.n.=>recept}
Let $\calf$ be a saturated fusion system over a finite $p$-group $S$, and 
let $\cale\le\calf$ be a fusion subsystem (not necessarily saturated) over 
$T\le S$. Let $\calh\subseteq\scrs(T)$ be a nonempty set of subgroups 
closed under $\cale$-conjugacy and overgroups in $T$, and assume that the 
following two properties hold for all $P\in\calh$: 
\begin{enumi} 

\item for each $\4P\le T$ containing $P$ and each $\varphi\in\homf(\4P,T)$, if 
$\varphi|_{P}\in\Hom_\cale(P,T)$, then 
$\varphi\in\Hom_\cale(\4P,T)$; and 

\item for each $\varphi\in\homf(N_T(P),S)$, there are $R\le S$ and 
$\psi\in\homf(R,T)$ such that $R\ge\gen{\varphi(N_T(P)),C_S(\varphi(P))}$ 
and $\psi\varphi\in\Hom_\cale(N_T(P),T)$. 

\end{enumi}
Then every subgroup $P\in\calh$ that is fully normalized or fully centralized 
in $\cale$ is also receptive (hence 
fully centralized) in $\cale$ and in $\calf$. If in addition, 
\begin{enumi}\setcounter{enumi}{2}
\item $\Inn(T)\in\sylp{\Aut_\cale(T)}$,
\end{enumi}
then $\cale$ is $\calh$-saturated.
\end{Lem}

\begin{proof} Fix $P\in\calh$ that is fully normalized or fully centralized 
in $\cale$, and choose $P_2\in P^\calf$ that is fully normalized in 
$\calf$. By \cite[Lemma I.2.6(c)]{AKO}, there is 
$\varphi\in\homf(N_T(P),S)$ such that $\varphi(P)=P_2$. By (ii), there is 
$\psi\in\homf(R,T)$ such that $\gen{\varphi(N_T(P)),C_S(P_2)}\le R\le S$ 
and $\psi\varphi\in\Hom_\cale(N_T(P),T)$. 

Set $P_3=\psi(P_2)\in P^\cale$; thus $\varphi(N_T(P))\le N_R(P_2)$ and 
$\psi(N_R(P_2))\le N_T(P_3)$. Likewise, $\varphi(C_T(P))\le C_R(P_2)$ and 
$\psi(C_R(P_2))\le C_T(P_3)$, where $C_R(P_2)=C_S(P_2)$ by assumption in (ii). 
Then $\psi\varphi(C_T(P))=\psi(C_S(P_2))=C_T(P_3)$ 
since $P$ is fully centralized or fully normalized in $\cale$, 
and $P$ and $P_3$ are both fully 
centralized in $\calf$ and in $\cale$ since $P_2$ is. 

We claim that $C_S(P)\le T$. By (ii), applied with the inclusion of 
$N_T(P)$ into $S$ in the role of $\varphi$, there are 
$R\ge\gen{N_T(P),C_S(P)}$ and $\omega\in\homf(R,T)$ such that 
$\omega|_{N_T(P)}\in\Hom_\cale(N_T(P),T)$. Then 
$\omega(C_S(P))=C_S(\omega(P))\le T$ since $P$ is fully centralized in 
$\calf$, and $\omega(C_T(P))=C_T(\omega(P))$ since $P$ is fully centralized 
in $\cale$ and $\omega|_{C_T(P)}\in\Mor(\cale)$. 
So $C_S(P)=C_T(P)$.

We have now shown that $P$ is fully centralized and hence receptive in 
$\calf$. It remains to show that it is receptive in $\cale$. Let $Q\in 
P^{\cale}\subseteq\calh$ and $\rho\in\Iso_{\cale}(Q,P)$ be arbitrary, and 
consider the subgroups 
	\begin{align*} 
	N_\rho^\calf &= \{ x\in N_S(Q) \,|\, \rho c_x^{Q} \rho^{-1} 
	\in \Aut_S(P) \} \\
	N_\rho^{\cale} &= \{ x\in N_{T}(Q) \,|\, \rho c_x^{Q}\rho^{-1} 
	\in \Aut_{T}(P) \} \le N_\rho^\calf. 
	\end{align*}
Since $P$ is receptive in $\calf$, $\rho$ extends to some 
$\4\rho\in \homf(N_\rho^\calf,S)$. For each $x\in N_\rho^\cale$, 
$c_{\4\rho(x)}^P=\rho c_x^{Q}\rho^{-1}\in\Aut_T(P)$, and hence there 
is $y\in T$ such that $\4\rho(x)\in yC_S(P)$. We just showed that 
$C_S(P)\le T$, and so $\4\rho(x)\in T$. 

Thus $\4\rho$ restricts to $\5\rho\in \Hom_\calf(N_\rho^\cale,T)$. 
By (i) and since $\5\rho|_{Q}=\rho\in\Mor(\cale)$, 
$\5\rho\in\Hom_\cale(N_\rho^\cale,T)$, extending $\rho$. Since 
$\rho$ was arbitrary, this shows that $P$ is receptive in $\cale$. 

If in addition, (iii) holds, and $T$ is fully automized in 
$\cale$, then $\cale$ is $\calh$-saturated by Proposition 
\ref{p:H-sat-gen}(b).
\end{proof}

We end the section with two group theoretic lemmas which are included for 
convenient reference. The first is very elementary. 

\begin{Lem} \label{l:sylp(G)}
Let $\chi\:G\too K$ be a homomorphism of finite groups, and let $P\le G$ be 
a $p$-subgroup. Then $P\in\sylp{G}$ if and only if 
$\chi(P)\in\sylp{\chi(G)}$ and $\Ker(\chi|_P)\in\sylp{\Ker(\chi)}$.
\end{Lem}

\begin{proof} Just note that  $|G:P| = |\chi(G):\chi(P)| \cdot 
|\Ker(\chi):\Ker(\chi|_P)|$.
\end{proof}

The following well known result about automorphisms of $p$-groups is useful 
when identifying elements of $O_p(\autf(P))$ in a fusion system $\calf$ 
over $S\ge P$.

\begin{Lem} \label{l:Op(AutP)}
Let $P$ be a finite $p$-group, and let $1=P_0\le P_1\le\cdots\le P_k=P$ be 
a sequence of subgroups, all normal in $P$. Let $\Gamma\le\Aut(P)$ be a 
group of automorphisms that normalizes each of the $P_i$. Then for each 
$\alpha\in\Gamma$ such that $x^{-1}\alpha(x)\in P_{i-1}$ for each $1\le 
i\le k$ and each $x\in P_i$, we have $\alpha\in O_p(\Gamma)$.
\end{Lem}

\begin{proof} Let $\Gamma_0\le\Gamma$ be the subgroup of all 
$\alpha\in\Gamma$ that induce the identity on each $P_i/P_{i-1}$. Then 
$\Gamma_0\nsg\Gamma$ since $\Gamma$ normalizes each of the $P_i$. It 
remains only to prove that each element of $\Gamma_0$ has $p$-power order, 
and this is shown, for example, in \cite[Lemma 5.3.3]{Gorenstein}. 
\end{proof}


\section{Maps from fusion systems to groups}

We give in this section a very general setup for constructing saturated 
fusion subsystems: one that includes Theorem \ref{ThA} as a special case. 
Recall that when $\calf$ is a fusion system over $S$ and 
$\calh\subseteq\scrs(S)$ is a set of subgroups of $S$, we let $\calf^\calh$ 
denote the full subcategory of $\calf$ whose set of objects is $\calh$. 
Thus $\Mor(\calf^\calh)$ is the set of all morphisms in $\calf$ between 
subgroups in $\calh$. 

\begin{Hyp} \label{h:chi-setup}
Let $\calf$ be a saturated fusion system over a finite $p$-group $S$. Fix a 
subgroup $T\nsg S$, a set of subgroups $\calh\subseteq\scrs(S)$ such that 
$T\in\calh$, and a map $\chi\:\Mor(\calf^\calh)\too G$ for some finite 
group $G$. Assume the following hold:
\begin{enumi} 

\item \label{h2-1} $\calh$ is closed under $\calf$-conjugacy and 
overgroups; 

\item \label{h2-2} $\chi(\varphi\psi)=\chi(\varphi)\chi(\psi)$ whenever 
$\varphi,\psi\in\Mor(\calf^\calh)$ are composable, and $\chi(\incl_P^Q)=1$ 
for each $P\le Q$ in $\calh$;

\item \label{h2-3} $\chi(\Inn(T))=1$ and $\chi(\autf(T))=G$; and 


\item \label{h2-4} for each $P\le T$ not in $\calh$, there is $x\in 
N_T(P)\sminus P$ such that $c_x^P\in O_p(\autf(P))$. 

\end{enumi}
\end{Hyp}

We could instead define $\chi$ as a functor from $\calf^{\calh}$ to 
$\calb(G)$, where $\calb(G)$ is a category with one object and endomorphism 
group $G$, and then remove the first part of condition \eqref{h2-2}. We 
could also have defined $\chi$ to be a homomorphism from 
$\pi_1(|\calf^\calh|)$ to $G$, where $|\calf^\calh|$ is the geometric 
realization of the category $\calf^\calh$, and then removed \eqref{h2-2} 
completely. (See, e.g., Section III.2.2 and Proposition III.2.8 in 
\cite{AKO} for more detail.) But it seems simplest to work with a map $\chi$ 
as above. 

\begin{Not} \label{n:chi-setup}
Assume Hypotheses \ref{h:chi-setup}. Set $U=\chi(\Aut_S(T))\le G$, and 
let $\5\chi\:S\too U$ be the homomorphism that sends $x\in S$ to 
$\chi(c_x^T)\in\chi(\Aut_S(T))=U$. For each $V\le U$, define 
	\[ S_V = \5\chi^{-1}(V) = \{x\in S \,|\, \5\chi(x)\in V \} 
	\quad\text{and}\quad 
	\calh\7V = \calh\cap\scrs(S_V) = \{P\in\calh \,|\, P\le S_V \}. \] 
For each $H\le G$, define the fusion subsystem $\calf_H\le\calf$ over 
$S_{U\cap H}$ by setting
	\[ \calf_H = \Gen{\varphi\in\homf(P,Q) \,\big|\, 
	P,Q\in\calh\7{U\cap H},~ \chi(\varphi)\in H }. \]
In particular, we set $S_{\Id}=\Ker(\5\chi)$, and let $\calf_{\Id}$ 
be the fusion system over $S_{\Id}$ generated by morphisms in $\chi^{-1}(1)$. 
\end{Not}

Note that $S_U=S$ in the situation of Hypotheses \ref{h:chi-setup} and 
Notation \ref{n:chi-setup}. We will see later that 
$\calh=\calh\7U\supseteq\calf^{cr}$, and hence that $\calf_G=\calf$ by 
Theorem \ref{t:AFT} (Alperin's fusion theorem).

We first list some of the basic properties that always hold in the 
situation of Hypotheses \ref{h:chi-setup}.

\begin{Lem} \label{l:chi-setup}
Assume Hypotheses \ref{h:chi-setup} and Notation \ref{n:chi-setup}. Then 
the following hold:
\begin{enuma} 

\item \label{l:chi-1} For all $P\in\calh$ and $x\in N_S(P)$, we have 
$\5\chi(x)=\chi(c_x^P)$. For $P,Q\in\calh$ and $\varphi\in\homf(P,Q)$, we 
have $\9{\chi(\varphi)}\5\chi(x)=\5\chi(\varphi(x))$ for all $x\in P$. 

\item \label{l:chi-2} For each $Q,R\le U$, we have 
$\chi(\homf(S_Q,S_R))=\{g\in G\,|\,\9gQ\le R\}$. Also, $U\in\sylp{G}$. 

\item \label{l:chi-3} If $P\in\calh$, then $C_S(P)\le S_{\Id}$. 

\item \label{l:chi-4} If $\cald\le\calf$ is a fusion subsystem 
over $D\le S$ with $D\ge T$, then $\cald^{cr}\subseteq\calh$. More 
precisely, for each subgroup $P\in\cald^c\sminus\calh$, 
	\[ \Out_D(P) \cap O_p(\Out_\cald(P)) \ne 1. \]

\item \label{l:chi-5} If $V\le U$ is strongly closed in $U$ with respect to 
$G$, then $S_V$ is strongly closed in $\calf$. 

\end{enuma}
\end{Lem}

\begin{proof} \noindent\textbf{(a) } For $P\in\calh$ and $x\in N_S(P)$, we 
have $\5\chi(x)=\chi(c_x^T)=\chi(c_x^{PT})=\chi(c_x^P)$, where the last two 
equalities hold since $\chi(\incl_T^{PT})=1=\chi(\incl_P^{PT})$.

For $P,Q\in\calh$, $x\in P$, and $\varphi\in\homf(P,Q)$, we have $\varphi 
c_x^P=c_{\varphi(x)}^Q\varphi\in\homf(P,Q)$, so 
$\9{\chi(\varphi)}\5\chi(x)=\5\chi(\varphi(x))$. 

\smallskip

\noindent\textbf{(b) } Fix $Q,R\le U$. 
Then $\chi(\homf(S_Q,S_R))\subseteq\{g\in G\,|\,\9gQ\le R\}$  
by (a), and it remains to show the opposite inclusion. Fix $g\in 
G$ such that $\9g Q\le R\le U$: we must find $\varphi\in\homf(S_Q,S_R)$ 
such that $\chi(\varphi)=g$. Choose $\beta\in\autf(T)$ such that 
$\chi(\beta)=g$. 

Since $\chi(\Aut_{S_Q}(T))=Q$, 
	\[ \chi(\9\beta\Aut_{S_Q}(T)) = \9{\chi(\beta)}Q = \9gQ\le R = 
	\chi(\Aut_{S_R}(T)), \] 
and hence $\9\beta\Aut_{S_Q}(T)\le\Aut_{S_R}(T)\Aut_{\calf_{\Id}}(T)$. 
Since $\Aut_{S_{\Id}}(T)=\Aut_S(T)\cap\Aut_{\calf_{\Id}}(T)
\in\sylp{\Aut_{\calf_{\Id}}(T)}$ 
(recall $\Aut_{\calf_{\Id}}(T)\nsg\autf(T)$), and $\Aut_{S_R}(T)$ 
normalizes $\Aut_{\calf_{\Id}}(T)$ and contains $\Aut_{S_{\Id}}(T)$, there 
is $\gamma\in\Aut_{\calf_{\Id}}(T)$ such that 
$\9{\gamma\beta}\Aut_{S_Q}(T)\le \Aut_{S_R}(T)$. Also, $T$ is receptive in 
$\calf$ since it is normal in $S$, so $\gamma\beta$ extends to 
$\varphi\in\homf(S_Q,S_R)$, and 
$\chi(\varphi)=\chi(\gamma)\chi(\beta)=1\cdot g=g$. 

Since $\calf$ is saturated and $T$ is fully normalized (normal in $S$), we 
have $\Aut_S(T)\in\sylp{\autf(T)}$. Hence $U=\chi(\Aut_S(T))\in\sylp{G}$ 
(Lemma \ref{l:sylp(G)}). 

\smallskip

\noindent\textbf{(c) } If $P\in\calh$ and $x\in C_S(P)$, then 
$\5\chi(x)=\chi(c_x^P)=\chi(\Id_P)=1$ by \eqref{l:chi-1}, so $x\in S_{\Id}$. 
Thus $C_S(P)\le S_{\Id}$. 

\smallskip

\noindent\textbf{(d) } Let $\cald\le\calf$ be a fusion subsystem over $D\le 
S$ with $D\ge T$, and assume $P\in\cald^c\sminus\calh$. Set $P_0=P\cap 
T$, and let $Q=\{t\in N_T(P_0)\,|\,c_t\in O_p(\autf(P_0))\}$. Then $P\le 
N_S(Q)$, and $Q\nleq P$ by Hypotheses \ref{h:chi-setup}\eqref{h2-4} and 
since $P_0\notin\calh$. So $N_{PQ}(P)>P$, and we can choose $x\in 
N_Q(P)\sminus P$. Then $c_x^P|_{P_0}\in O_p(\autf(P_0))$ and $c_x^P$ 
induces the identity on $P/P_0$, so $c_x^P\in O_p(\autf(P))$ by Lemma 
\ref{l:Op(AutP)}. 

Thus $c_x^P\in O_p(\autf(P))\cap\Aut_\cald(P)\le O_p(\Aut_\cald(P))$. Also, 
$P\ge C_D(P)$ since $P$ is $\cald$-centric, and $[c_x^P]\ne1$ in $\Out(P)$ 
since $x\in T\le D$. So $c_x^P\notin\Inn(P)$, and $[c_x^P]\ne1$ in 
$\Out_D(P)\cap O_p(\Out_\cald(P))$. In particular, $P$ is not 
$\cald$-radical, and since this holds for all $P\in\cald^c\sminus\calh$, we 
conclude that $\cald^{cr}\subseteq\calh$. 

\smallskip

\noindent\textbf{(e) } Assume $V\le U$ is strongly closed in $U$ with 
respect to $G$. If $S_V$ is not strongly closed in $\calf$, then by Theorem 
\ref{t:AFT} (Alperin's fusion theorem), there are $P\in\calf^{cr}$, 
$\alpha\in\autf(P)$, and $x\in P\cap S_V$ such that $\alpha(x)\notin S_V$. 
Also, $P\in\calh$ by \eqref{l:chi-4}, applied with $\calf$ in the role of 
$\cald$. Then $\5\chi(x)\in V$ and 
$\5\chi(\alpha(x))\notin V$ are conjugate in $G$ by \eqref{l:chi-1}, 
contradicting the assumption that $V$ is strongly closed.
\end{proof}

We next look at some of the properties of the fusion subsystems $\calf_H$. 

\begin{Lem} \label{l:H<G}
Assume Hypotheses \ref{h:chi-setup} and Notation \ref{n:chi-setup}, fix a 
subgroup $H\le G$, and set $V=H\cap U$. 
\begin{enuma} 

\item \label{l:H<G-1} If $P\in\calh\7V$ is  fully normalized or fully 
centralized in $\calf_H$, then $P$ is receptive and fully centralized in 
$\calf_H$ and in $\calf$. 

\item \label{l:H<G-2} We have 
$\calf_H^{cr}\subseteq\calf_H^c\cap\calh \subseteq\calf^c$.

\end{enuma}
\end{Lem}

\begin{proof} \noindent\textbf{(a) } By Lemma \ref{l:f.n.=>recept}, this 
implication is true if the following two conditions hold.
\begin{enumi} 

\item \boldd{For each $P\le\4P$ in $\calh$ and $\varphi\in\homf(\4P,T)$, if 
$\varphi|_P\in\Hom_{\calf_H}(P,T)$, then 
$\varphi\in\Hom_{\calf_H}(\4P,T)$.}\vskip-15pt%
\noindent This holds since for each such $\varphi$, 
$\chi(\varphi)=\chi(\varphi)\chi(\incl_P^{\4P})=\chi(\varphi|_P)\in H$.

\item \boldd{For each $P\in\calh$ and $\varphi\in\homf(N_{S_V}(P),S)$, there 
are $R\le S$ and $\psi\in\homf(R,S_V)$, where 
$R\ge\gen{\varphi(N_{S_H}(P)),C_S(\varphi(P))}$ and 
$\psi\varphi\in\Hom_{\calf_H}(N_{S_V}(P),S_V)$.} To see this, set 
$g=\chi(\varphi)$: then $\chi(\varphi(N_{S_V}(P)))\le\9gV\cap U$ by Lemma 
\ref{l:chi-setup}\eqref{l:chi-1}. Set $R=S_{\9gV\cap U}$: by Lemma 
\ref{l:chi-setup}\eqref{l:chi-2}, there is $\psi\in\homf(R,S_V)$ such that 
$\chi(\psi)=g^{-1}$. Thus $\psi\varphi\in\Hom_{\calf_H}(N_{S_V}(P),S_V)$. 
Also, $\varphi(N_{S_V}(P))\le R$ by construction, and $C_S(\varphi(P))\le 
S_{\Id}\le R$ by Lemma \ref{l:chi-setup}\eqref{l:chi-3}.

\end{enumi}

\smallskip

\noindent\textbf{(b) } If $P\in\calf_H^{cr}$, then $P\in\calh$ by Lemma 
\ref{l:chi-setup}\eqref{l:chi-4}, while $P\in\calf_H^c$ by definition. So 
$\calf_H^{cr}\subseteq\calf_H^c\cap\calh$. If $P\in\calf_H^c\cap\calh$, 
then $P$ is fully centralized in $\calf_H$ since it is $\calf_H$-centric, 
and is fully centralized in $\calf$ by \eqref{l:H<G-1}. 
Also, $C_S(P)=C_{S_{V}}(P)\le P$ by Lemma \ref{l:chi-setup}\eqref{l:chi-3}, 
and hence $P\in\calf^c$. 
\end{proof}

We are now ready to determine under what conditions the fusion subsystems 
$\calf_H$ are saturated. 

\begin{Thm} \label{t:H<G}
Assume Hypotheses \ref{h:chi-setup} and Notation \ref{n:chi-setup}. For 
each subgroup $H\le G$, 
\begin{enuma} 

\item the fusion subsystem $\calf_H$ is saturated if and only if $H\cap 
U\in\sylp{H}$; and 

\item $\calf_H$ is saturated and normal in $\calf$ if and only if 
$H\nsg G$. 

\end{enuma}
\end{Thm}

\begin{proof} Set $V=H\cap U$ for short. If $\calf_H$ is 
saturated, then $\Aut_{S_V}(T)\in\sylp{\Aut_{\calf_H}(T)}$, and hence 
$V=\chi(\Aut_{S_V}(T))\in\sylp{H}$ by Lemma \ref{l:sylp(G)}.

Conversely, assume now that $V\in\sylp{H}$; we must show that $\calf_H$ is 
saturated. In Step 1, we reduce to the case where $S_V$ is fully normalized 
in $\calf$. In Step 2, we prove that $\calf_H$ is $\calh\7V$-saturated, and 
in Step 3, finish the proof that it is saturated. Point (b) is shown in 
Step 4. 


\smallskip

\noindent\textbf{Step 1: } We first show that it suffices to prove 
saturation of $\calf_H$ when $S_V$ is fully normalized in $\calf$. Let 
$H\le G$ be arbitrary, assuming only that $V\in\sylp{H}$. Choose 
$Q\in(S_V)^\calf$ that is fully normalized in $\calf$ and 
$\varphi\in\isof(S_V,Q)$, and set $g=\chi(\varphi)$. Then $Q\ge S_{\Id}$ since 
$S_{\Id}$ is strongly closed (Lemma \ref{l:chi-setup}\eqref{l:chi-5}), so 
$Q=S_W$ where $W=\5\chi(Q)=\9gV$ (Lemma 
\ref{l:chi-setup}\eqref{l:chi-1}). Then $\9\varphi\calf_H=\calf_{\9gH}$: 
by inspection, $c_\varphi$ sends the defining generators of $\calf_H$ (see 
Notation \ref{n:chi-setup}) to those of $\calf_{\9gH}$. In particular, 
$\calf_H$ is saturated if $\calf_{\9gH}$ is saturated, and so it suffices 
to prove the latter. Note that $W=\9gV\le \9gH\cap U$, with equality since 
$\9gV\in\sylp{\9gH}$. 

We assume from now on that $S_V=S_W$ is fully normalized in $\calf$. 

\smallskip

\noindent\textbf{Step 2: } We now show that $\calf_H$ is 
$\calh\7V$-saturated. By Lemma \ref{l:H<G}\eqref{l:H<G-1}, if 
$P\in\calh\7V$ is fully normalized in $\calf_H$, then it is also receptive 
in $\calf_H$. So by Proposition \ref{p:H-sat-gen}(b), it remains only to 
prove that $S_V$ is fully automized in $\calf_H$; i.e., that 
$\Inn(S_V)\in\sylp{\Aut_{\calf_H}(S_V)}$. 

By Lemma \ref{l:chi-setup}\eqref{l:chi-2}, the homomorphism  
	\[ \chi_V \: \autf(S_V) \Right4{} N_G(V)/V 
	\qquad\textup{defined by}\qquad \chi_V(\alpha)=\chi(\alpha)V \]
is well defined and surjective. Also, 
$\Aut_{\calf_H}(S_V)=\chi_V^{-1}(N_H(V)/V)$ by definition, and $N_H(V)/V$ 
has order prime to $p$ since $V\in\sylp{H}$. Thus $\Ker(\chi_V)$ has index 
prime to $p$ in $\Aut_{\calf_H}(S_V)$. Furthermore, 
$\Aut_S(S_V)\in\sylp{\autf(S_V)}$ since $S_V$ is fully normalized in 
$\calf$, so by Lemma \ref{l:sylp(G)}, $\Ker(\chi_V|_{\Aut_S(S_V)})$ is a 
Sylow $p$-subgroup of $\Ker(\chi_V)$ and hence of $\Aut_{\calf_H}(S_V)$. 
For $x\in N_S(S_V)$, 
	\[ c_x^{S_V}\in\Ker(\chi_V) ~\iff~ \chi(c_x^{S_V})\in V
	~\iff~ \5\chi(x)\in V ~\iff~ x\in S_V, \]
so that $\Ker(\chi_V|_{\Aut_S(S_V)})=\Inn(S_V)$, finishing the proof that 
$\Inn(S_V)\in\sylp{\Aut_{\calf_H}(S_V)}$.

\smallskip

\noindent\textbf{Step 3: } By Lemma \ref{l:chi-setup}\eqref{l:chi-4}, we  
have $\Out_{S_{V}}(P)\cap O_p(\Out_{\calf_H}(P))\ne1$ for each 
$P\in\calf_H^c\sminus\calh$. Also, $\calf_H$ is $\calh\7V$-saturated by 
Step 2, and is $\calh\7V$-generated by definition of $\calf_H$ (see 
Notation \ref{n:chi-setup}). So $\calf_H$ is saturated by Proposition 
\ref{p:H-sat-gen}. 

\smallskip

\noindent\textbf{Step 4: } If $H$ is normal in $G$, then $V=H\cap 
U\in\sylp{H}$ since $U\in\sylp{G}$, so $\calf_H$ is saturated by Step 3. 

By Lemma \ref{l:chi-setup}\eqref{l:chi-5}, $S_V$ is strongly closed in 
$\calf$. Each $\alpha\in\autf(S_{V})$ permutes the generating set used to 
define $\calf_H$, and hence induces an automorphism of $\calf_H$. So the 
invariance condition for $\calf_H\le\calf$ holds. The extension condition 
also holds since $C_S(S_V)\le S_{\Id}\le S_V$ by Lemma 
\ref{l:chi-setup}\eqref{l:chi-3}. 

Now, $G=N_G(V)H$ by the Frattini argument. Fix $P\in\calh\7{V}$ and 
$\varphi\in\homf(P,S_V)$, and let $h\in H$ and $g\in N_G(V)$ be such that 
$\chi(\varphi)=gh$. By Lemma \ref{l:chi-setup}\eqref{l:chi-2}, there is 
$\beta\in\autf(S_V)$ such that $\chi(\beta)=g$. Set 
$\varphi_0=\varphi\beta^{-1}$; then $\chi(\varphi_0)=h\in H$ and hence 
$\varphi_0\in\Hom_{\calf_H}(P,S_V)$. This proves the Frattini 
condition for $\calf_H\le\calf$ for morphisms between members of $\calh$, 
and the general case follows by Theorem \ref{t:AFT} and since 
$(\calf_H)^{cr}\subseteq\calh$ (Lemma \ref{l:chi-setup}\eqref{l:chi-4}). 
Thus $\calf_H\nsg\calf$. 

Conversely, assume $\calf_H$ is saturated and normal in $\calf$. Then $T\le 
S_V$ for each $V\le U$ since $\5\chi(T)=1$ by Hypotheses 
\ref{h:chi-setup}\eqref{h2-3}, and hence $T$ is an object in $\calf_H$. 
(This holds for all $H\le G$.) Since $\calf_H\nsg\calf$, we have 
$\Aut_{\calf_H}(T)\nsg\autf(T)$ (see, e.g., \cite[Proposition 
I.6.4(c)]{AKO}), and hence $H=\chi(\Aut_{\calf_H}(T))\nsg 
\chi(\autf(T))=G$. 
\end{proof}

\begin{Rmk} \label{rmk:linking}
For readers familiar with linking systems associated to fusion systems and 
their geometric realizations (see Sections III.2 and III.4 in \cite{AKO}), 
we note here that in the situation of Hypotheses \ref{h:chi-setup} and 
Notation \ref{n:chi-setup}, if $\call$ is a centric linking system 
associated to $\calf$, then it is straightforward to define a linking 
subsystem $\call_H\le\call$ associated to $\calf_H$ for each $H\le G$ such 
that $H\cap U\in\sylp{H}$. More precisely, if $\pi\:\call\too\calf$ is the 
structure functor for $\call$, then we set 
$\til\chi=\chi\circ\pi\:\Mor(\call^\calh)\too G$, where $\call^\calh$ is 
the full subcategory of $\call$ with objects in $\calh\cap\calf^c$, and let 
$\call_H\le\call$ be the subcategory with objects $\calf^c\cap\calh\7V$ 
($V=H\cap U$) and morphisms in $\til\chi^{-1}(H)$. Using the fact that 
$(\calf_H)^{cr}\subseteq\calf^c$ (see Lemma \ref{l:H<G}\eqref{l:H<G-2}), 
one easily checks that this is a linking system. 

This setup also gives information about the fundamental groups of the 
geometric realizations $|\calf^\calh|$ and $|\call|$. As noted earlier, the 
map $\chi\:\Mor(\calf^\calh)\too G$ can be regarded as a surjective 
homomorphism from $\pi_1(|\calf^\calh|)$ onto $G$. Since $|\call|$ and 
$|\call^\calh|$ are homotopy equivalent, where $\call^\calh\subseteq\call$ 
is the full subcategory with object set $\calh$ (see \cite[Theorem 
3.5]{BCGLO1}), $\til\chi$ induces a surjection from 
$\pi_1(|\call|)\cong\pi_1(|\call^\calh|)$ onto $G$, and the spaces 
$|\call_H|$ (for $H\le G$ as above) are equivalent to covering spaces of 
$|\call|$ (see \cite[Propsition III.2.9]{AKO}). 
\end{Rmk}


\section{The normalizer of a set of components}

We adopt here the notation and terminology used in \cite{O-KRS} for 
morphisms of fusion systems and sets of commuting subsystems. Recall
that we set $\kk=\{1,\dots,k\}$ for $k\ge1$. 

A sequence $\cale_1,\dots,\cale_k\le\calf$ of fusion subsystems over finite 
$p$-groups $T_1,\dots,T_k\le S$ \emph{commutes in $\calf$} if there is a 
morphism of fusion systems from $\xxx\cale$ to $\calf$ that extends the 
inclusions of the $\cale_i$ into $\calf$. In other words, there is a 
homomorphism of groups $I\:\xxx{T}\too S$ that sends $(t_1,\dots,t_k)$ to 
$t_1\cdots t_k$, and which induces a functor $\5I$ from the category 
$\xxx\cale$ to $\calf$. Equivalently:

\begin{Lem}[{\cite[Lemma 2.8]{O-KRS}}] \label{l:comm.subsys.}
Let $\calf$ be a saturated fusion system over a finite $p$-group $S$, and 
let $\cale_1,\dots\cale_k\le\calf$ be fusion subsystems over subgroups 
$T_1,\dots,T_k\le S$. Then $\cale_1,\dots\cale_k$ commute in $\calf$ if and 
only if $[T_i,T_j]=1$ for each $i\ne j$ in $\kk$, and for each $k$-tuple of 
morphisms $\varphi_i\in\Hom_{\cale_i}(P_i,Q_i)$ (for $i\in\kk$), there is a 
morphism $\varphi\in\homf(P_1\cdots P_k,Q_1\cdots Q_k)$ such that 
$\varphi|_{P_i}=\varphi_i$ for each $i$.
\end{Lem}

When $\cale_1,\dots,\cale_k\le\calf$ commute in $\calf$, and 
$\5I\:\xxx\cale\too\calf$ is as above, we define the \emph{central 
product} of the $\cale_i$ in $\calf$ to be the fusion subsystem 
\begin{small} 
	\beqq \begin{split} 
	\cale_1&\cdots\cale_k = \5I(\xxx\cale) \\
	&= \Gen{ \varphi\in\homf(P,Q) \,\big|\, P=P_1\cdots P_k,~ 
	Q=Q_1\cdots Q_k,~ \varphi|_{P_i}\in\Hom_{\cale_i}(P_i,Q_i) ~\forall 
	i\in\kk } .
	\end{split} \label{e:E1..Ek} \eeqq
\end{small}%
The central product of commuting saturated fusion subsystems is always 
saturated: we leave it as an exercise to show that 
$\cale_1\cdots\cale_k\cong(\xxx\cale)/Z$ for some central subgroup 
$Z\le Z(\xxx\cale)$. (The image of an arbitrary morphism between saturated 
fusion systems is also saturated, but this is a much deeper theorem first 
shown by Puig: see, e.g., Proposition 5.11 and Corollary 5.15 in 
\cite{Craven}.)

Recall the fusion subsystems $Z\cale\le\calf$, for $\cale\le\calf$ and 
$Z\le Z(\calf)$, of Definition \ref{d:F/Z}(b). 

\begin{Lem} \label{l:E/Z=prod}
Let $\calf$ be a saturated fusion system over a finite $p$-group $S$, and 
let $\cale_1,\dots,\cale_k\le\calf$ be commuting saturated fusion 
subsystems over $T_1,\dots,T_k\le S$. Set $T=T_1\cdots T_k$ and 
$\cale=\cale_1\cdots\cale_k$, and set $Z=Z(\cale)$. Then 
\begin{enuma} 


\item for each $i\in\kk$, $\cale_i$ is the full subcategory of $\cale$ with 
objects the subgroups of $T_i$; 

\item if $t_i\in T_i$ for all $i\in\kk$ and $t=t_1\cdots t_k$, then 
$t^\cale=\{u_i\cdots u_k\,|\,u_i\in t_i^{\cale_i} \}$; 


\item $T/Z=(T_1Z/Z)\times\cdots\times(T_kZ/Z)$; and 


\item $\cale/Z = (Z\cale_1/Z)\times\cdots\times (Z\cale_k/Z)$ where 
$Z\cale_i/Z\cong\cale_i/Z(\cale_i)$ for each $i$. 

\end{enuma}
\end{Lem}

\begin{proof} By assumption, there is a functor $\5I$ from 
$\cale_1\times\cdots\times\cale_k$ into $\calf$ that restricts to the 
inclusion on each factor. Since each $\cale_i$ is a full subcategory of the 
direct product $\prod_{i=1}^k\cale_i$, its image under $\5I$ is a 
full subcategory of $\5I(\cale_1\times\cdots\times\cale_k)=\cale$. This 
proves (a). 

Point (b) follows immediately from \eqref{e:E1..Ek}. Since 
the center of a saturated fusion system is the set of elements whose 
conjugacy class has order 1 (see \cite[Lemma I.4.2]{AKO}), (b) implies that 
	\beqq \textup{for $t_i\in T_i$ (all $i\in\kk$), $t_1\cdots t_k\in 
	Z$ if and only if $t_i\in Z(\cale_i)$ for all $i$.} 
	\label{e:Z1..Zk} \eeqq

Let $I\:\xxx{T}\too T$ be the homomorphism $I(t_1,\dots,t_k)=t_1\cdots 
t_k$. Thus $I$ is surjective, and 
$I^{-1}(Z)=Z(\cale_1)\times\cdots\times Z(\cale_k)$
by \eqref{e:Z1..Zk}. This proves (c), and 
also shows that $I$ induces an isomorphism from 
$\prod_{i=1}^k(T_i/Z(\cale_i))$ to $T/Z$. It also proves that $Z\cap 
T_i=Z(\cale_i)$ for each $i$, and hence that the inclusion of $T_i$ into 
$ZT_i$ induces an isomorphism $T_i/Z(\cale_i)\cong ZT_i/Z\le T/Z$. Thus 
$T/Z$ is the direct product of its subgroups $ZT_i/Z$, and the independence 
of the factors $(T_iZ/Z)$ implies that $T_iZ\cap T_jZ=Z$ for any pair $i\ne 
j$ of distinct indices in $\kk$.

The equality $\cale/Z=\prod_{i=1}^k(Z\cale_i/Z)$ now follows immediately 
from the description of $\cale=\cale_1\cdots\cale_k$ in \eqref{e:E1..Ek}.
\end{proof}

We now describe the setup needed to state our main theorem on normalizers. 

\begin{Hyp} \label{h:Ei<F}
Let $\calf$ be a saturated fusion system over a finite $p$-group $S$, and 
let $\cale_1,\dots,\cale_k\le \calf$ (for some $k\ge1$) be saturated 
fusion subsystems over $T_1,\dots,T_k\le S$ that commute in $\calf$. 
Set $T=T_1\cdots T_k$ and $\cale=\cale_1\cdots\cale_k$, and assume 
\begin{enumi}

\item \label{h3-2} $\cale\nsg\calf$; 

\item \label{h3-3} $T_i\nleq Z(\cale)$ for each $i\in\kk$; and 

\item \label{h3-4} each element of $\autf(T)$ permutes the subgroups 
$T_i$ for $i\in\kk$. 

\end{enumi}
\end{Hyp}


By Lemma \ref{l:E/Z=prod}(a), each $\cale_i$ is the full subcategory of 
$\cale$ with objects the subgroups of $T_i$. Since each 
$\alpha\in\autf(T)$ permutes the subgroups $T_i$ by \eqref{h3-4}, this means that $\alpha$ 
also permutes the fusion subsystems $\cale_i$.

We also need some more notation. 

\begin{Not} \label{n:Ei<F}
Let $\cale_1,\dots,\cale_k\le\calf$ be as in Hypotheses \ref{h:Ei<F}. Let 
	$\chi_0\: \autf(T) \too \Sigma_k$
be the homomorphism defined by \eqref{h3-4}: $\chi_0(\alpha)=\sigma$ (for 
$\alpha\in\autf(T)$ and $\sigma\in\Sigma_k$) if 
$\alpha(T_iZ(\cale))=T_{\sigma(i)}Z(\cale)$ for each $i\in\kk$. Define also 
	\begin{align*} 
	Z&=Z(\cale), & 
	G&=\chi_0(\autf(T))\le\Sigma_k, \\
	\calh &= \{ P\le S \,|\, P\cap T_i \nleq Z 
	~\textup{for each $i\in\kk$} \}, &
	U&=\chi_0(\Aut_S(T))\le G. 
	\end{align*}
\end{Not}

In the next proposition, we give some other criteria that imply condition 
\eqref{h3-4} in Hypotheses \ref{h:Ei<F}. A fusion system is 
\emph{indecomposable} if it is not the direct product of two or more proper 
fusion subsystems. Recall also the focal subgroup $\foc(-)$ from Definition 
\ref{d:subgroups}(f).

\begin{Prop} \label{p:(iv')=>(iv)}
Assume $\cale_1,\dots,\cale_k\le\calf$ are saturated fusion systems over 
finite $p$-groups $T_1,\dots,T_k\le S$ that commute in $\calf$ and satisfy 
conditions \eqref{h3-2} and \eqref{h3-3} in Hypotheses \ref{h:Ei<F}. Then 
either of the conditions 
\begin{enumi} 

\item[\rm(\ref{h3-4}$'$) ] for each $i\in\kk$, $\cale_i/Z(\cale_i)$ is 
indecomposable and $\foc(\cale_i)=T_i$, or 

\item [\rm(\ref{h3-4}$''$) ] for each $i\in\kk$, $\cale_i/Z(\cale_i)$ is 
indecomposable, $T_i\ge Z(\cale_1\cdots\cale_k)$, and 
$Z(\cale_i/Z(\cale_i))=1$ 

\end{enumi}
implies condition \eqref{h3-4}. 
\end{Prop}

\begin{proof} Set $\cale=\cale_1\cdots\cale_k$, $T=T_1\cdots T_k$, and 
$Z=Z(\cale)\le Z(T)$. By Lemma \ref{l:E/Z=prod}(c,d), we have 
$\cale/Z=\prod_{i=1}^k(Z\cale_i/Z)$ and $T/Z=\prod_{i=1}^k(ZT_i/Z)$, where 
for each $i$, $Z\cale_i/Z\cong\cale_i/Z(\cale_i)$ and hence is 
indecomposable under either condition (\ref{h3-4}$'$) or (\ref{h3-4}$''$). 
Also, $\foc(\cale/Z)=T/Z$ (if (\ref{h3-4}$'$) holds) or $Z(\cale/Z)=1$ (if 
(\ref{h3-4}$''$) holds), and in either case, the factorization of $\cale/Z$ 
as a product of the indecomposable fusion subsystems $Z\cale_i/Z$ is unique 
by \cite[Corollary 5.3]{O-KRS}. 

Now assume $\alpha\in\Aut_\calf(T)\le\Aut(\cale)$. Then $\alpha$  induces 
an automorphism $\4\alpha\in\Aut(\cale/Z)$, and by the uniqueness of the 
factorization, there is $\sigma\in\Sigma_k$ such that 
$\4\alpha(T_iZ/Z)=T_{\sigma(i)}Z/Z$ and hence $\alpha(T_iZ)=T_{\sigma(i)}Z$ 
for each $i$. If (\ref{h3-4}$''$) holds, then $T_iZ=T_i$ for each $i$, so 
$\alpha$ permutes the $T_i$. If (\ref{h3-4}$'$) holds, then $\alpha$ 
permutes the fusion subsystems $Z\cale_i$ (since they are full 
subcategories of $\cale$ by Lemma \ref{l:E/Z=prod}(a)), and hence permutes 
the focal subgroups $\foc(Z\cale_i)=\foc(\cale_i)=T_i$. In either case, 
$\alpha(T_i)=T_{\sigma(i)}$, finishing the proof of \eqref{h3-4}. 
\end{proof}

The following definition of a component of a fusion system is taken 
from \cite{A-gfit}.

\begin{Defi} \label{d:component}
Let $\cale\le\calf$ be saturated fusion systems over finite $p$-groups 
$T\le S$. The subsystem $\cale$ is \emph{subnormal} in $\calf$ (denoted 
$\cale\snsg\calf$) if there is a sequence of subsystems 
$\cale=\cale_0\nsg\cale_1\nsg\cdots\nsg\cale_m=\calf$ each normal in the 
following one. The subsystem $\cale$ is \emph{quasisimple} if 
$\cale/Z(\cale)$ is simple and $\foc(\cale)=T$. The \emph{components} of 
$\calf$ are the subnormal saturated fusion subsystems of $\calf$ that are 
quasisimple. 
\end{Defi}

A saturated fusion system $\calf$ is \emph{constrained} if it has a normal 
subgroup $Q\nsg\calf$ that is $\calf$-centric. By (9.9.1) and (9.12.3) in 
\cite{A-gfit}, $\calf$ is constrained if and only if it has no components. 
The assumption in the following example that $\calf$ not be 
constrained is made to ensure that it has at least one component.

\begin{Ex} \label{p:Comp(F)}
Let $\calf$ be a saturated fusion system over a finite $p$-group $S$, and 
assume $\calf$ is not constrained. Then $\calf$ satisfies Hypotheses 
\ref{h:Ei<F}, with the set $\Comp(\calf)$ of components of $\calf$ in the 
role of $\{\cale_1,\dots,\cale_k\}$ and with their central product 
in the role of $\cale$. 
\end{Ex}

\begin{proof} All of these statements are shown in \cite[Chapter 
9]{A-gfit}. Since $\calf$ is not constrained, $C_S(O_p(\calf))\nleq 
O_p(\calf)$, so $O_p(\calf)$ is strictly contained in the generalized 
Fitting subsystem $F^*(\calf)$ by \cite[9.11]{A-gfit}, and hence 
$\Comp(\calf)\ne\emptyset$ by \cite[9.9]{A-gfit}. Also, $E(\calf)$ (defined 
in \cite{A-gfit} to be the smallest normal subsystem containing all 
components) is the central product of the components by 
\cite[9.9.1]{A-gfit}, and is normal in $\calf$ by \cite[9.8.1]{A-gfit}. 
Thus the components of $\calf$ commute and satisfy condition 
\eqref{h3-2}, and satisfy \eqref{h3-3} and (\ref{h3-4}$'$) since they 
are quasisimple. 
\end{proof}

We now start to look at some of the consequences of these hypotheses. 
Condition (a) in the next lemma is stated in two forms: a simpler form 
which suffices here in most cases, and a longer, more technical form needed 
in one of the proofs below. 

\begin{Lem} \label{l:Ei<F}
Assume Hypotheses \ref{h:Ei<F} and Notation \ref{n:Ei<F}. 
For $I\subseteq\kk$, set $T_I=\gen{T_i\,|\,i\in I}$. 
\begin{enuma} 

\item \label{l:Ei-1} For each $P,Q\le S$ and $\varphi\in\homf(P,Q)$, there 
is an element $\sigma\in G\le\Sigma_k$ such that $\varphi(P\cap T_i)\le 
T_{\sigma(i)}$ for each $i\in\kk$. If $P\in\calh$, then $Q\in\calh$, and 
this element $\sigma$ is unique. 

\item[\rm(a$'$) ] Fix $P,Q\le S$ and $\varphi\in\homf(P,Q)$. Then there are 
$\sigma\in G\le\Sigma_k$ and $\alpha\in\autf(T)$ such that $\varphi(P\cap 
T_I)\le T_{\sigma(I)}=\alpha(T_I)$ for each $I\subseteq\kk$. In particular, 
$\varphi(P\cap T_i)\le T_{\sigma(i)}\in(T_i)^\calf$ for each $i\in\kk$. If 
$j\in\kk$ is such that $P\cap T_j\nleq Z$, then 
$\varphi(P\cap T_j)\nleq T_\ell$ for $\ell\ne\sigma(j)$. 

\end{enuma}
Define $\chi\:\Mor(\calf^\calh)\too G$ by setting $\chi(\varphi)=\sigma$ 
whenever $P,Q\in\calh$ and $\varphi$ and $\sigma$ are as in \eqref{l:Ei-1}. 
Then 
\begin{enuma} \setcounter{enumi}{1}
\item \label{l:Ei-2} $\calf$, $T$, $\calh$, $\chi$, $G$, 
and $U$ together satisfy Hypotheses \ref{h:chi-setup}. 
\end{enuma}
\end{Lem} 

\begin{proof} Recall that $Z=Z(\cale)$ and $\cale\nsg\calf$, and
$\calh=\{P\le S\,|\, P\cap T_i\nleq Z(\cale)~\forall\,i\in\kk\}$.

\smallskip


\noindent\textbf{(a,a$'$) } Fix $P,Q\le S$ and $\varphi\in\homf(P,Q)$. Then 
$\varphi(P\cap T)\le T$ since $T$ is strongly closed in $\calf$ (recall 
$\cale\nsg\calf$). By the 
Frattini condition for $\cale\nsg\calf$, there are $\alpha\in\autf(T)$ and 
$\varphi_0\in\Hom_\cale(P\cap T,T)$ such that $\varphi|_{P\cap 
T}=\alpha\varphi_0$. By Hypotheses \ref{h:Ei<F}\eqref{h3-4}, there is 
$\sigma=\chi_0(\alpha)\in G\le\Sigma_k$ such that 
$\alpha(T_i)=T_{\sigma(i)}$ for each $i$. Hence 
for each $I\subseteq\kk$, 
	\beqq \varphi(P\cap T_I)=\alpha(\varphi_0(P\cap T_I)) \le 
	\alpha(T_I) = T_{\sigma(I)}. \label{l:Ei-3} \eeqq

Now let $j,\ell\in\kk$ be such that $P\cap T_j\nleq Z$ and $\varphi(P\cap 
T_j)\le T_\ell$. Choose $x\in(P\cap T_j)\sminus Z$. Then 
$\varphi(x)Z\in(T_{\sigma(j)}Z/Z)\cap(T_{\ell}Z/Z)$, and $\varphi(x)\notin 
Z$ since $Z=Z(\cale)$ is strongly closed in $\calf$. So $\ell=\sigma(j)$ by 
Lemma \ref{l:E/Z=prod}(c). 

In particular, if $P\in\calh$, then $P\cap T_i\nleq Z$ for each $i\in\kk$, 
so $\sigma$ is the unique permutation satisfying $\varphi(P\cap T_i)\le 
T_{\sigma(i)}$ for each $i$. Also, $\varphi(P)\cap 
T_{\sigma(i)}\ge\varphi(P\cap T_i)\nleq Z$ for each $i$ since $Z$ is 
strongly closed, so $Q\ge\varphi(P)\in\calh$.

\smallskip

\noindent\textbf{(b) } We must prove the following four properties:
\begin{enumi} 

\item $\calh$ is closed under $\calf$-conjugacy and overgroups; 

\item $\chi(\varphi\psi)=\chi(\varphi)\chi(\psi)$ whenever 
$\varphi,\psi\in\Mor(\calf^\calh)$ are composable, and $\chi(\incl_P^Q)=1$ 
for each $P\le Q$ in $\calh$;

\item $\chi(\Inn(T))=1$ and $\chi(\autf(T))=G$; and 

\item for each $P\le T$ not in $\calh$, there is $x\in N_T(P)\sminus P$ 
such that $c_x^P\in O_p(\autf(P))$. 

\end{enumi}
Point \eqref{h2-1} holds since $\calh$ is closed 
under $\calf$-conjugacy by \eqref{l:Ei-1}, and is closed under overgroups 
by definition. Point \eqref{h2-2} and the first statement in \eqref{h2-3} 
hold by definition of $\chi$, and $\chi(\autf(T))=G$ by definition of $G$ 
(Notation \ref{n:Ei<F}) and since $\chi|_{\autf(T)}=\chi_0$. 

It remains to show \eqref{h2-4}. Let $Y\le T$ be such that 
$Y\ge Z$ and $Y/Z=Z(T/Z)$. Thus $Y\nsg T$. For each $i\in\kk$, 
$(Y\cap T_iZ)/Z=Z(T_iZ/Z)\ne1$ (recall that $T/Z=\prod_{i=1}^k(T_iZ/Z)$ by 
Lemma \ref{l:E/Z=prod}(c)), so $Y\cap T_iZ\nleq Z$, and $Y\cap T_i\nleq Z$ 
since $Y\ge Z$. Hence $Y\in\calh$.

If $P\le T$ and $P\notin\calh$, then $P\ngeq Y$ since $\calh$ is closed 
under overgroups. Hence $YP>P$ and $N_{YP}(P)>P$. Choose $x\in 
N_Y(P)\sminus P$; then $[x,Z\cap P]=1$ and $[x,P]\le Z\cap P$. So 
$c_x^P$ induces the identity on $P\cap Z$ and on $P/(P\cap Z)$, and since 
$Z$ is strongly closed in $\calf$, we have $c_x^P\in O_p(\autf(P))$ by 
Lemma \ref{l:Op(AutP)}. 
\end{proof}

We need some more notation. Note that as a special case of Hypotheses 
\ref{h:Ei<F}\eqref{h3-4}, the 
conjugation action of each $x\in S$ permutes the subgroups $T_i\le T$ for 
$i\in\kk$, thus inducing an action of $S$ on $\kk$. 

\begin{Not} \label{n:EJ<F}
Assume Hypotheses \ref{h:Ei<F} and Notation \ref{n:Ei<F}. For each nonempty 
subset $J=\{j_1,\dots,j_\ell\}\subseteq\kk$, define 
	\begin{align*} 
	T_J&=T_{j_1}\cdots T_{j_\ell} = \gen{T_j\,|\,j\in J}, & 
	N_J&=N_S(J)=N_S(T_J), \\
	\cale_J &= \cale_{j_1}\cdots\cale_{j_\ell}, &
	W_J&=C_S(J) = \bigcap\nolimits_{j\in J}N_S(T_j). 
	\end{align*}
Also, define 
	\begin{align*} 
	\calh\8J &= \{P\le S \,|\, P\cap T_j\nleq Z~ \textup{for 
	each $j\in J$}\} \supseteq \calh, \\
	\caln_J &= \Gen{\varphi\in\homf(P,Q) \,\big|\, 
	P,Q\in\calh\8J\6{N_J},~ \varphi(P\cap T_J)\le T_J },
	\quad\textup{and} \\
	\calw_J &= \Gen{\varphi\in\homf(P,Q) \,\big|\, 
	P,Q\in\calh\8J\6{W_J},~ \varphi(P\cap T_j)\le T_j~ \textup{for each 
	$j\in J$} }. 
	\end{align*}
\end{Not}

These fusion subsystems $\calw_J\le\caln_J$ over $W_J\le N_J$ are the main 
focus of our attention in the rest of the section. We will show in Theorem 
\ref{t:NJ} that they are saturated, and are unchanged if we replace 
$\calh\8J$ by $\calh$ in their definitions. 


\begin{Lem} \label{l:EJ<F}
Assume Hypotheses \ref{h:Ei<F} and Notation \ref{n:Ei<F} and \ref{n:EJ<F}. 
Then for each $\emptyset\ne J\subseteq\kk$, 
\begin{enuma} 

\item \label{l:EJ-1} if $\chi_0(\alpha)(J)=J$ for each $\alpha\in\autf(T)$, 
then $\cale_J\nsg\calf$; and 

\item \label{l:EJ-2} if $\cald\le\calf$ is a fusion subsystem over $D\le S$ 
such that $\cale_J\nsg\cald$,  then $\cald^{cr}\subseteq\calh\8J$.

\end{enuma}
\end{Lem}

\begin{proof} \noindent\textbf{(a) } We first show that $\cale_J\nsg\cale$ 
(for all $\emptyset\ne J\subseteq\kk$). Lemma \ref{l:E/Z=prod}(b) implies 
that $T_J$ is strongly closed in $\cale$. The invariance and Frattini 
conditions for $\cale_J\le\cale$ hold since $\cale_J$ is the full 
subcategory of $\cale$ with objects the subgroups of $T_J$ (Lemma 
\ref{l:E/Z=prod}(a)). By Lemma \ref{l:comm.subsys.}, for each 
$\alpha\in\Aut_{\cale_J}(T_J)$, there is $\4\alpha\in\Aut_\cale(T)$ such 
that $\4\alpha|_{T_J}=\alpha$ and $\4\alpha|_{T_i}=\Id_{T_i}$ for each 
$i\in\kk\sminus J$, and hence $[\4\alpha,C_T(T_J)]=[\alpha,Z(T_J)]\le 
Z(T_J)$. The extension condition thus holds, and so $\cale_J\nsg\cale$. 

Thus $\cale_J\nsg\cale\nsg\calf$ (recall $\cale\nsg\calf$ by Hypotheses 
\ref{h:Ei<F}\eqref{h3-2}). If $\autf(T)\le\Sigma_k$ sends $J$ to itself, 
then each $\alpha\in\autf(T)$ permutes the $T_j$ for $j\in J$ by Hypotheses 
\ref{h:Ei<F}\eqref{h3-4}, hence normalizes $T_J$, and normalizes $\cale_J$ 
since it is the full subcategory of $\cale$ with objects the subgroups of 
$T_J$ (Lemma \ref{l:E/Z=prod}(a) again). So $\cale_J\nsg\calf$ by 
\cite[7.4]{A-gfit}. 

\smallskip

\noindent\textbf{(b) } Let $J\subseteq\kk$ be arbitrary. Set 
$Z_J=Z(\cale_J)$, and let $Y_J\le T_J$ be such that $Y_J\ge Z_J$ and 
$Y_J/Z_J=Z(T_J/Z_J)$. For each $j\in J$, $(Y_J\cap 
T_jZ_J)/Z_J=Z(T_jZ_J/Z_J)\ne1$ since $T_jZ_J/Z_J\cong T_jZ/Z\ne1$ by 
Hypotheses \ref{h:Ei<F}\eqref{h3-3} (and since the $T_j$ commute pairwise). 
Let $t_j\in T_j$ and $z_j\in Z_J$ be such that $t_jz_j\in Y_J\sminus 
Z_J$. Then $t_j\in Y_J\sminus Z_J$ since $z_j\in Z_J\le Y_J$, and 
$t_j\in(Y_J\cap T_j)\sminus Z$ since $T_j\cap Z\le Z_J$. Thus $Y_J\cap 
T_j\nleq Z$ for each $j\in J$, and so $Y_J\in\calh\8J$.

Let $\cald\le\calf$ be a fusion subsystem over $D\le S$ such that 
$\cale_J\nsg\cald$, and in particular, $T_J\nsg D$. Let $P\le D$ be such 
that $P\notin\calh\8J$. Then $P\ngeq Y_J$ since $\calh\8J$ is closed under 
overgroups. Also, $P$ normalizes $Y_J$ since it normalizes $T_J$ and $Z_J$, 
so $PY_J>P$ and $N_{PY_J}(P)>P$. 

Choose $x\in N_{Y_J}(P)\sminus P$. Then $[x,T_J]\le[Y_J,T_J]\le Z_J$ and 
$[x,Z_J]\le[Y_J,Z_J]=1$ (recall $Z_J=Z(\cale_J)\le Z(T_J)$), so $c_x^P$ 
acts trivially on $P\cap Z_J$ and on $(P\cap T_J)/(P\cap Z_J)$. Also, 
$[x,P]\le[N_{Y_J}(P),P]\le P\cap T_J$ since $Y_J\le T_J\nsg D$ and $P\le 
D$, so $c_x^P$ also acts trivially on $P/(P\cap T_J)$. Since 
$\cale_J\nsg\cald$, the subgroups $T_J$ and $Z_J$ are both strongly closed 
in $\cald$, and hence $c_x^P\in O_p(\Aut_\cald(P))$ by Lemma 
\ref{l:Op(AutP)}. 

Recall that $x\in Y_J\le T_J\le D$. 
If $c_x^P\notin\Inn(P)$, then $O_p(\Out_\cald(P))\ne1$ and $P$ is not 
$\cald$-radical. If $c_x^P\in\Inn(P)$, then $x\in PC_D(P)\sminus P$, so 
$C_D(P)\nleq P$, and $P$ is not $\cald$-centric. In either case, 
$P\notin\cald^{cr}$, and hence $\cald^{cr}\subseteq\calh\8J$. 
\end{proof}

We are now ready to prove the ``normalizing'' properties of the fusion 
subsystems $\calw_J\le\caln_J\le\calf$ defined in Notation \ref{n:EJ<F}. 


\begin{Thm} \label{t:NJ}
Assume Hypotheses \ref{h:Ei<F} and Notations \ref{n:Ei<F} and \ref{n:EJ<F}, 
and let $J\subseteq\kk$ be a nonempty subset such that $T_J$ is fully 
normalized in $\calf$. Then $\calw_J$ and $\caln_J$ are saturated fusion 
subsystems of $\calf$, $\calw_J$ is normal in $\caln_J$, and the following 
hold.
\begin{enuma}

\item For all $P,Q\in\calh\8J$ contained in $N_J$, 
	\beqq \Hom_{\caln_J}(P,Q) = \bigl\{  
	\varphi\in\homf(P,Q) \,\big|\, \varphi(P\cap T_J)\le T_J \bigr\}. 
	\label{t:NJ-1a} \eeqq
Also, $\cale_J\nsg\caln_J$. If $\cald\le\calf$ is a saturated 
fusion subsystem such that $\cale_J\nsg\cald$, then $\cald\le\caln_J$. 

\item For all $P,Q\in\calh\8J$ contained in $W_J$, 
	\beqq \Hom_{\calw_J}(P,Q) = \bigl\{  
	\varphi\in\homf(P,Q) \,\big|\, \varphi(P\cap T_j)\le T_j~ 
	\textup{for each $j\in J$} \bigr\}. \label{t:NJ-1b} \eeqq 
Also, $\cale_j\nsg\calw_J$ for each $j\in J$. If $\cald\le\calf$ is 
a saturated fusion subsystem such that $\cale_j\nsg\cald$ for 
each $j\in J$, then $\cald\le\calw_J$. 

\end{enuma}
\end{Thm}

\newcommand{\0}{^0}   

\begin{proof} We first check \eqref{t:NJ-1a} and \eqref{t:NJ-1b}. 
Each of the morphisms appearing in the definition of $\caln_J$ (see Notation 
\ref{n:EJ<F}) sends elements of $T_J$ to elements of $T_J$. So $T_J$ is 
strongly closed in $\caln_J$. Hence for all $P,Q\in\calh\8J\6{N_J}$, the 
set $\Hom_{\caln_J}(P,Q)$ is contained in the right hand side of 
\eqref{t:NJ-1a}, while the opposite inclusion holds by definition of 
$\caln_J$. This proves \eqref{t:NJ-1a}, and \eqref{t:NJ-1b} follows by a 
similar argument.

Define 
	\begin{align*} 
	\caln\0_J &= \Gen{\varphi\in\homf(P,Q) \,\big|\, 
	P,Q\in\calh\6{N_J},~ \varphi(P\cap T_J)\le T_J } \le \caln_J, \\
	\calw\0_J &= \Gen{\varphi\in\homf(P,Q) \,\big|\, 
	P,Q\in\calh\6{W_J},~ \varphi(P\cap T_j)\le T_j~ \textup{for each 
	$j\in J$} } \le \calw_J. 
	\end{align*}
Thus the only difference between $\caln\0_J$ and $\caln_J$ or between 
$\calw\0_J$ and $\calw_J$ lies in the set of objects used in the definition. 

We prove in Step 1 that $\calw\0_J$ and $\caln\0_J$ are saturated and 
$\calw\0_J\nsg\caln\0_J$. In Step 2, we prove (a) and (b) when $J=\kk$ (in 
which case $\calh=\calh\8J$ and hence $\caln\0_J=\caln_J$ and 
$\calw\0_J=\calw_J$). We then prove $\caln\0_J=\caln_J$ and (a) for 
arbitrary $J\subseteq\kk$ in Steps 3 and 4, respectively, and prove 
$\calw\0_J=\calw_J$ and (b) in the general case in Step 5. Throughout the 
proof, we assume that $T_J$ is fully normalized in $\calf$.

\smallskip

\noindent\textbf{Step 1: } Let $\chi_0\:\autf(T)\too G$ and 
$U=\chi_0(\Aut_S(T))\le G$ be as in Notation \ref{n:Ei<F}, and let 
$\5\chi\:S\too U$ be the surjective homomorphism that sends $x\in S$ to 
$\chi_0(c_x^T)\in\chi_0(\Aut_S(T))=U$. By Lemma \ref{l:Ei<F}\eqref{l:Ei-2}, 
we are in the situation of Hypotheses \ref{h:chi-setup}. 

Consider the action of $G\le\Sigma_k$ on $\kk$, and set 
	\[ H=N_G(J),\quad H_0=C_G(J),\quad V=H\cap U=N_U(J),\quad 
	V_0=H_0\cap U=C_U(J). \]
Note that $\caln\0_J=\calf_H$ and $\calw\0_J=\calf_{H_0}$ under Notation 
\ref{n:chi-setup}.

Recall that $N_J=N_S(T_J)=N_S(J)$ under the action of $S$ on $\kk$ induced 
by its conjugation action on $T$. For each $\sigma\in G$, there is 
$\alpha\in\autf(T)$ such that $\chi_0(\alpha)=\sigma$, and 
$T_{\sigma(J)}=\alpha(T_J)\in(T_J)^\calf$. Since $T_J$ is fully normalized 
in $\calf$, we have $|N_S(T_{\sigma(J)})|\le|N_S(T_J)|$ 
and hence $|N_S(\sigma(J))|\le|N_S(J)|$ for each $\sigma\in G$. The map 
$N_S(\sigma(J))\too N_U(\sigma(J))$ induced by $\5\chi$ is surjective and 
its kernel $\Ker(\5\chi)=C_S(\kk)$ is independent of $\sigma$, so 
$|N_U(\sigma(J))|\le|N_U(J)|$ for all $\sigma\in G$. Since $N_U(\sigma(J)) 
= U\cap N_G(\sigma(J))$ where $N_G(\sigma(J))=\9\sigma N_G(J)=\9\sigma H$ 
for each $\sigma$, this shows that $|U\cap H|\ge|U\cap\9\sigma 
H|=|U^\sigma\cap H|$ for each $\sigma\in G$, and hence that  $V=U\cap 
H\in\sylp{H}$. Since $H_0\nsg H$, we also have $V_0=V\cap H_0\in 
\sylp{H_0}$. 

The hypotheses of Theorem \ref{t:H<G} thus hold, and so $\caln\0_J=\calf_H$ 
and $\calw\0_J=\calf_{H_0}$ are both saturated. By the same theorem applied 
with $\caln\0_J$ in the role of $\calf$ (and since $H_0\nsg H$), we have 
$\calw\0_J\nsg\caln\0_J$.

\smallskip

\noindent\textbf{Step 2: } Assume $J=\kk$. In particular, 
$\calh\8{J}=\calh$, $\caln\0_J=\caln_J$, and $\calw\0_J=\calw_J$. Point (a) 
holds since $\cale_{\kk}=\cale\nsg\calf=\caln_{\kk}$ by Hypotheses 
\ref{h:Ei<F}\eqref{h3-2}. 

It remains to prove (b). We already checked \eqref{t:NJ-1b}, and 
$\cale_i\nsg\calw_{\kk}$ for each $i\in\kk$ by Lemma 
\ref{l:EJ<F}\eqref{l:EJ-1}, applied with $\calw_{\kk}$ in the role of 
$\calf$ (hence with $G=1$) and $\{i\}$ in the role of $J$. 

Let $\cald\le\calf$ be a saturated fusion subsystem over $D\le S$ such 
that $\cale_i\nsg\cald$ for each $i\in\kk$. Then $T_i\nsg D$ for each 
$i\in\kk$, so $D\le W_{\kk}$ and $D\ge T_1\cdots T_k=T$. Also, for each 
$P,Q\in\calh\6{D}$, 
	\[ \Hom_\cald(P,Q) \le \{\varphi\in\homf(P,Q) \,|\, \varphi(P\cap 
	T_i)\le T_i ~\forall\, i\in\kk \} = \Hom_{\calw_{\kk}}(P,Q) \]
(recall $T_i$ is strongly closed in $\cald$). 
Since $\cald^{cr}\subseteq\calh$ by Lemma \ref{l:EJ<F}\eqref{l:EJ-2}, 
we now have $\cald\le\calw_{\kk}$ by Theorem \ref{t:AFT} (Alperin's fusion 
theorem).

\smallskip

\noindent\textbf{Step 3: } Throughout the rest of the proof, we let 
$\emptyset\ne J\subseteq\kk$ be arbitrary, subject to the condition that 
$T_J$ be fully normalized in $\calf$. We prove in this step that 
$\caln\0_J=\caln_J$.

We first check that $\caln_J$ is $\calh\8J\6{N_J}$-saturated. By 
Lemma \ref{l:f.n.=>recept}, it suffices to show, for all 
$P\in\calh\8J\6{N_J}$, that the following three points hold.
\begin{enumi} 

\item \boldd{If $P\le\4P\le N_J$, and $\varphi\in\homf(\4P,N_J)$ is such 
that $\varphi|_P\in\Hom_{\caln_J}(P,N_J)$, then $\varphi\in\Mor(\caln_J)$.} 
By Lemma \ref{l:Ei<F}(\ref{l:Ei-1}$'$), there is $\sigma\in\Sigma_k$ such 
that $\varphi(\4P\cap T_i)\le T_{\sigma(i)}$ for each $i\in\kk$ and 
$\varphi(\4P\cap T_J)\le T_{\sigma(J)}$. Since $P\in\calh\8J$, there is 
$x_j\in(P\cap T_j)\sminus Z$ for each $j\in J$, and since 
$\varphi|_P\in\Mor(\caln_J)$, we have $\varphi(x_j)\in(T_{\sigma(j)}\cap 
T_J)\sminus Z$. So $\sigma(j)\in J$ by Lemma \ref{l:E/Z=prod}(c), and 
$\sigma(J)=J$. Hence $\varphi(\4P\cap T_J)\le T_J$, and so 
$\varphi\in\Mor(\caln_J)$.

\item \boldd{For each $\varphi\in\homf(N_{N_J}(P),S)$, there are $R\le S$ and 
$\psi\in\homf(R,N_J)$ such that 
$R\ge\gen{\varphi(N_{N_J}(P)),C_S(\varphi(P))}$ and 
$\psi\varphi\in\Hom_{\caln_J}(N_{N_J}(P),N_J)$.} 
For such $\varphi$, by 
Lemma \ref{l:Ei<F}(\ref{l:Ei-1}$'$), there is $\sigma\in G\le\Sigma_k$ such 
that $\varphi(N_{N_J}(P)\cap T_i)\le T_{\sigma(i)}$ for each $i\in\kk$ and 
$\varphi(N_{N_J}(P)\cap T_J)\le T_{\sigma(J)}\in(T_J)^\calf$. 
Set $J'=\sigma(J)$. 
Since $T_J$ is fully normalized in $\calf$, there 
is $\psi\in\homf(N_S(T_{J'}),S)$ such that $\psi(T_{J'})=T_J$ (see 
\cite[Lemma I.2.6(c)]{AKO}) and hence $\psi(N_S(T_{J'}))\le N_S(T_J)=N_J$. 
Then $\psi\varphi(N_{N_J}(P)\cap T_J)\le T_J$, and so 
$\psi\varphi\in\Hom_{\caln_J}(N_{N_J}(P),N_J)$. Also, $C_S(\varphi(P))\le 
C_S(J')\le N_S(J')=N_S(T_{J'})$ since $\varphi(P)\in\calh\8{J'}$. 

It remains to show that $\varphi(N_{N_J}(P))\le N_S(T_{J'})$. Fix $x\in 
N_{N_J}(P)$. Then $c_x^T$ permutes the subgroups $P\cap T_j$ for $j\in J$, 
and $P\cap T_j\nleq Z$ since $P\in\calh\8J$. So $c_{\varphi(x)}^T$ 
permutes the subgroups $\varphi(P\cap T_j)\le T_{\sigma(j)}$. Hence 
$\5\chi(\varphi(x))\in N_U(J')$, and $\varphi(x)\in N_S(J')=N_S(T_{J'})$. 
Thus $\varphi(N_{N_J}(P))\le N_S(T_{J'})$.

\item \boldd{The Sylow group $N_J$ is fully automized in $\caln_J$.} This holds 
since $\Aut_{\caln_J}(N_J)=\Aut_{\caln\0_J}(N_J)$ by definition, and 
$\caln\0_J$ (also over $N_J$) has already been shown to be saturated. 

\end{enumi} 

Thus $\caln_J$ is $\calh\8J\6{N_J}$-saturated. Since 
$(\caln_J)^{cr}\subseteq\calh$ by Lemma \ref{l:chi-setup}\eqref{l:chi-4}, 
and $\caln_J$ is defined so as to be $\calh\8J\6{N_J}$-generated, $\caln_J$ 
is saturated by Proposition \ref{p:H-sat-gen}(a). Also, $\caln_J$ and 
$\caln\0_J$ are equal after restriction to subgroups in 
$\calh\6{N_J}\supseteq(\caln_J)^{cr}$, so $\caln\0_J=\caln_J$ by Theorem 
\ref{t:AFT} (Alperin's fusion theorem). 

\smallskip

\noindent\textbf{Step 4: } We now prove point (a) in the general case. 
We first check that $\cale\nsg\caln_J$. 
By the extension condition for $\cale\nsg\calf$, each 
$\alpha\in\Aut_\cale(T)$ extends to some $\4\alpha\in\autf(TC_S(T))$ such 
that $[\alpha,C_S(T)]\le Z(T)$. Also, $C_S(T)\le C_S(\kk)\le N_S(J)=N_J$, 
so $TC_S(T)=TC_{N_J}(T)$. Since $\cale$ is a central product of the 
$\cale_i$, we have $\alpha(T_i)=T_i$ for each $i\in\kk$, so in particular, 
$\4\alpha(T_J)=\alpha(T_J)=T_J$, and 
$\4\alpha\in\Aut_{\caln_J}(TC_{N_J}(T))$. The extension condition thus 
holds for $\cale\le\caln_J$, and so $\cale\nsg\caln_J$ by Lemma 
\ref{l:Fnormal}. 

Thus Hypotheses \ref{h:Ei<F} hold with $\cale\nsg\caln_J$ in the 
place of $\cale\nsg\calf$. So $\cale_J\nsg\caln_J$ by Lemma 
\ref{l:EJ<F}\eqref{l:EJ-1}.

Now assume $\cald\le\calf$ is a saturated fusion subsystem over $D\le S$ 
such that $\cale_J\nsg\cald$. In particular, $T_J\le D$ and is strongly 
closed in $\cald$, and hence $D\le N_S(T_J)=N_J$. 
By \eqref{t:NJ-1a} and since $T_J$ is strongly closed in $\cald$, 
for each $P\in\calh\8J\6D$, we have 
	\[ \Aut_\cald(P) \le \{\alpha\in\autf(P) \,|\, \alpha(P\cap T_J)\le 
	T_J \} = \Aut_{\caln_J}(P) . \]
Also, $\cald$ is $\calh\8J\6D$-generated by Theorem \ref{t:AFT} and since 
$\cald^{cr}\subseteq\calh\8J$ (Lemma \ref{l:EJ<F}\eqref{l:EJ-2}), and hence 
$\cald\le\caln_J$. 

\smallskip

\noindent\textbf{Step 5: } By construction, 
$\calw\0_J\le\calw_J\le\caln_J=\caln\0_J$, where $\caln\0_J$ and $\calw\0_J$ 
are saturated by Step 1 and the equality holds by Step 3. By Step 2, applied 
with $\cale_J\nsg\caln_J$ in the role of $\cale\nsg\calf$ and 
$\calh\8J\6{W_J}$ in the role of $\calh$, the subsystem $\calw_J$ is 
saturated, $\cale_j\nsg\calw_J$ for each $j\in J$, and $\cald\le\calw_J$ 
for each $\cald\le\calf$ such that $\cale_j\nsg\cald$ for all $j\in J$. 

To prove (b), it remains only to show that $\calw_J=\calw\0_J$. Both are 
saturated, and they are equal after restriction to subgroups in 
$\calh\6{W_J}$. Since $(\calw_J)^{cr}\subseteq\calh$ by Lemma 
\ref{l:chi-setup}\eqref{l:chi-4}, we have $\calw_J=\calw\0_J$ by Theorem 
\ref{t:AFT} (Alperin's fusion theorem). 
\end{proof}

More generally, for each set $\scrj$ of pairwise disjoint subsets of $\kk$ 
with union $\5J\subseteq\kk$, if $T_{\5J}$ is fully normalized in $\calf$ 
and certain ``Sylow conditions'' hold (those needed to apply Theorem 
\ref{t:H<G}), then one can construct a largest saturated fusion subsystem 
that normalizes $\cale_J$ for each $J\in\scrj$. However, the precise 
statement of such a result seems much more complicated than in the special 
cases considered in Theorem \ref{t:NJ}, so we won't show that here.

\end{document}